\documentclass{amsart}
\usepackage{amsfonts}

\usepackage{graphicx}
\usepackage{graphicx,esint}
\usepackage{amssymb,amscd,amsthm,amsxtra}
\usepackage{latexsym}
\usepackage{epsfig}
\usepackage{graphicx,color}

\newtheorem{thm}{Theorem}[section]
\newtheorem{cor}[thm]{Corollary}
\newtheorem{lem}[thm]{Lemma}
\newtheorem{prop}[thm]{Proposition}
\theoremstyle{definition}
\newtheorem{defn}[thm]{Definition}
\theoremstyle{remark}
\newtheorem{rem}[thm]{Remark}
\numberwithin{equation}{section}

\newcommand{\R}{\mathbb R}
\newcommand{\be}{\begin{equation}}
\newcommand{\ee}{\end{equation}}

\newcommand{\eps}{\varepsilon}

\newcommand{\p}{\partial}

\newcommand{\comment}[1]{}

\begin{document}
\title[ ]{Regularity of the free boundary for two-phase problems governed by
divergence form equations and applications}
\author{Daniela De Silva}
\address{Department of Mathematics, Barnard College, Columbia University,
New York, NY 10027}
\email{\texttt{desilva@math.columbia.edu}}
\author{Fausto Ferrari}
\address{Dipartimento di Matematica dell' Universit\`a, Piazza di Porta S.
Donato, 5, 40126 Bologna, Italy.}
\email{\texttt{fausto.ferrari@unibo.it}}
\author{Sandro Salsa}
\address{Dipartimento di Matematica del Politecnico, Piazza Leonardo da
Vinci, 32, 20133 Milano, Italy.}
\email{\texttt{sandro.salsa@polimi.it }}
\thanks{D.D. is partially supported by NSF grant DMS-1301535. F.~ F.~ is
supported by the ERC starting grant project 2011 EPSILON (Elliptic PDEs and
Symmetry of Interfaces and Layers for Odd Nonlinearities) 277749 and by RFO
grant, Universit\`a di Bologna, Italy. }

\begin{abstract}
We study a class of two-phase inhomogeneous free boundary problems governed
by elliptic equations in divergence form. In particular we prove that
Lipschitz or flat free boundaries are $C^{1,\gamma}$. Our results apply to
the classical Prandtl-Bachelor model in fluiddynamics.
\end{abstract}

\maketitle

% ----------------------------------------------------------------
% ----------------------------------------------------------------

%\tableofcontents

\section{Introduction and Statements of the Main Theorems}

This paper is a further step in the development of the theory for general
elliptic inhomogeneous two-phase free boundary problems, after \cite{DFS1}, 
\cite{DFS2}, \cite{DFS5}. In particular, in \cite{DFS5}, via Perron's
method, we constructed a Lipschitz viscosity solution to problems governed
by elliptic equations in divergence form with H\"older continuous
coefficients and we proved weak measure theoretical regularity properties,
such as ``flatness" of the free boundary in a neighborhood of each point of
its reduced part. Here, as in \cite{DFS1, DFS2} we prove that flat or
Lipschitz free boundaries are locally $C^{1,\gamma }$. It is worthwhile to
notice that, in absence of distributed sources and with Lipschitz
coefficients, these regularity results were obtained in \cite{FS1}, \cite%
{FS2}, while they are new even in the homogeneous case when the coefficients
are assumed to be merely H\"older continuous.\smallskip

Our setting is the following. Let $\Omega $ be a bounded Lipschitz domain in 
$\mathbb{R}^{n}$ and let $A=\{a_{ij}(x)\}_{1\leq i,j\leq n}$ be a symmetric matrix with H\"{o}%
lder continuous coefficients in $\Omega ,$ $A\in C^{0,\bar{\gamma}}(\Omega )$%
, which is uniformly elliptic, i.e. 
\begin{equation*}
\lambda \mid \xi \mid ^{2}\leq \sum_{i,j=1}^{n}a_{ij}(x)\xi _{i}\xi _{j}\leq
\Lambda \mid \xi \mid ^{2},\quad \forall x\in \Omega ,\quad \xi \in \mathbb{R%
}^{n}
\end{equation*}%
for some $0<\lambda \leq \Lambda .$ Denote by 
\begin{equation*}
\mathcal{L}:=\text{div}(A(x)\nabla \cdot ).
\end{equation*}%
Let $f\in L^{\infty }(\Omega )$. We consider the two-phase inhomogeneous
free boundary problem 
\begin{equation}
\left\{ 
\begin{array}{ll}
\mathcal{L}u=f & \text{ \ in }\Omega ^{+}(u)=\{u>0\} \\ 
\mathcal{L}u=f & \text{ \ in }\Omega ^{-}(u)=\{u\leq 0\}^{\circ } \\ 
|\nabla _{A}u^{+}|^{2}-|\nabla _{A}u^{-}|^{2}=1 & \text{ \ on }F(u)=\partial
\{u>0\}\cap \Omega ,%
\end{array}%
\right.  \label{FBintro}
\end{equation}%
where $|\nabla _{A}u|^{2}:=\langle A\nabla u,\nabla u\rangle $.

Since our emphasis is on the class of operators, we decided to avoid further
technicalities by considering only a particular, although significant, free
boundary condition. The extension to a general free boundary condition of
the type $|\nabla u^{+}|$ $=G(|\nabla u^{-}|,\nu ,x),$ where $\nu =\nu (x)$
denotes the unit normal to $F(u)$ at $x$ pointing towards $\Omega ^{+}(u),$
can be achieved without much difficulty as in \cite{DFS1}, if $G(\beta
,x,\nu )$ is strictly increasing in $\beta $, Lipschitz continuous in the
first and in the third argument, H\"{o}lder continuous in the second
argument, $G(0):=\inf_{x\in \Omega ,\left\vert \nu \right\vert =1}G(0,x,\nu
)>0,$ and moreover $\eta ^{-N}G(\eta ,x,\nu )$ is strictly decreasing in $%
\eta$ uniformly in $x, \nu $.

We now recall the notion of viscosity solution. Here we give it in terms of
test functions. In the last section we will use an equivalent notion in
terms of asymptotic developments at one side regular points of the free
boundary.

\begin{defn}
Given $u, \varphi \in C(\Omega)$, we say that $\varphi$ touches $u$ by below
(resp. above) at $x_0 \in \Omega$ if $u(x_0)= \varphi(x_0),$ and 
\begin{equation*}
u(x) \geq \varphi(x) \quad (\text{resp. $u(x) \leq \varphi(x)$}) \quad \text{%
in a neighborhood $O$ of $x_0$.}
\end{equation*}
If this inequality is strict in $O \setminus \{x_0\}$, we say that $\varphi$
touches $u$ strictly by below (resp. above).
\end{defn}

\begin{defn}
\label{defnhsol} Let $u$ be a continuous function in $\Omega$. We say that $%
u $ is a viscosity solution to (\ref{FBintro}) in $\Omega$, if the following
conditions are satisfied:

\begin{enumerate}
\item $\mathcal{L}u = f$ in $\Omega^+(u) \cup \Omega^-(u)$ in the weak sense;

\item Let $x_0 \in F(u)$ and $v \in C^{1, \bar \gamma}(\overline{B^+(v)}) \cap
C^{1, \bar \gamma}(\overline{B^-(v)})$ ($B=B_\delta(x_0)$) with $F(v) \in C^{2}$.
If $v$ touches $u$ by below (resp.above) at $x_0 \in F(v)$, then 
\begin{equation*}
|\nabla_A v^+(x_0)|^2 - |\nabla_A v^-(x_0)|^2 \leq 1 \quad (\text{resp. $%
\geq $)}.
\end{equation*}
\end{enumerate}
\end{defn}

We also need the definition of comparison subsolution (resp. supersolution).

\begin{defn}
\label{defsubcv} We say that $v \in C(\Omega)$ is a $C^{1,\bar{\gamma}}$ strict
(comparison) subsolution (resp. supersolution) to (\ref{FBintro}) in $\Omega$%
, if $v\in C^{1,\bar \gamma}(\overline{\Omega^+(v)}) \cap C^{1,\bar \gamma}(\overline{%
\Omega^-(v)})$, $F(v) \in C^2$, and the following conditions are satisfied:

\begin{enumerate}
\item $\mathcal{L}v> f $ (resp. $< f $) in $\Omega^+(v) \cup \Omega^-(v)$ in
the weak sense;

\item If $x_0 \in F(v)$, then 
\begin{equation*}
|\nabla_A v^+(x_0)|^2 - |\nabla_A v^-(x_0)|^2 > 1\quad (\text{resp. $%
|\nabla_A v^+(x_0)|^2 - |\nabla_A v^-(x_0)|^2 < 1$.)}
\end{equation*}
\end{enumerate}
\end{defn}

We notice that, using the almost monotonicity formula in \cite{MP}, one can
reproduce the proof of Theorem 4.5 in \cite{CJK} to prove that viscosity
solutions to \eqref{FBintro} are locally Lipschitz continuous.

Our main Theorem is a ``flatness implies regularity" result. Here, a
constant depending (possibly) on $n,Lip(u),\lambda ,\Lambda ,[a_{ij}]_{C^{0,%
\bar{\gamma}}},\Vert f\Vert _{L^{\infty }},$ is called universal.

\begin{thm}[Flatness implies $C^{1,\protect\gamma }$]
\label{flatmain2} Let $u$ be a  viscosity solution to $(\ref{FBintro}%
)$ in $B_{1}$. There exists a universal constant $\bar{\delta}>0$ such that, if 
\begin{equation*}
\{x_{n}\leq -\delta \}\subset B_{1}\cap \{u^{+}(x)=0\}\subset \{x_{n}\leq
\delta \},
\end{equation*}%
with $0\leq \delta \leq \bar{\delta},$ then $F(u)$ is $C^{1,\gamma }$ in $%
B_{1/2}$ for some universal $\gamma\in (0,1)$.
\end{thm}

The strategy to prove Theorem \ref{flatmain2} follows the lines of our work 
\cite{DFS1}. The key tools are a Harnack type inequality and an improvement
of flatness lemma which allow to linearize the problem into a standard
transmission problem.

Next, a ``Lipschitz implies regularity" result.

\begin{thm}[Lipschitz implies $C^{1,\protect\gamma}$]
\label{Lipmainvar} Let $u$ be a viscosity solution to $(\ref%
{FBintro})$ in $B_1$. If $F(u)$ is Lipschitz in $B_1$, then $F(u)$ is $C^{1,\gamma}$ in $B_{1/2}$ for some universal $\gamma\in (0,1)$.
\end{thm}

Theorem \ref{Lipmainvar} follows from our flatness result via a blow-up
argument and the regularity result in \cite{C1} for the homogeneous problem
when $A\equiv I$.

Actually, exploiting the variational nature of the free boundary condition,
we can use a Weiss type monotonicity formula \cite{W}, which together with
the monotonicity formula in \cite{ACF} provides a new proof of the
regularity result for the homogeneous problem and the Laplace operator. A
similar strategy has been used in \cite{DS}.

We remark however, that for general free boundary conditions one has to rely
on the result in \cite{C1}.\smallskip

A consequence of our flatness theorem is a regularity result for the
minimal Perron solution $u,$ constructed in \cite{DFS5}. We recall that in 
\cite{DFS5} we prove that $u$ is Lipschitz continuous with non-degenerate
positive part and the free boundary $F(u)$ has finite $(n-1)$ dimensional
Haursdorff measure. Moreover, for $c,r_{0}$ universal, $r<r_{0}$, we have 
\begin{equation*}
\mathcal{H}^{n-1}(F(u)\cap B_{r}(x))\leq cr^{n-1},\text{ for all }x\in F(u),
\end{equation*}%
and, denoting by $F^{\ast }(u)$ the reduced free boundary, 
\begin{equation*}
\mathcal{H}^{n-1}(F^{\ast }(u)\cap B_{r}(x))\geq cr^{n-1},\text{ }\mathcal{H}%
^{n-1}(F(u)\setminus F^{\ast }(u))=0.
\end{equation*}%
From Theorem \ref{flatmain2} we deduce the following result.

\begin{thm}
\label{precompactenss_th} Let $u$ be the Perron solution. In a neighborhood
of every $x_{0}\in F^{\ast }(u),$ $F(u)$ is a $C^{1,\gamma }$ surface.
\end{thm}

Important questions remain open as further regularity results, that we will
consider in a forthcoming paper, and the analysis of the singular points of $%
F(u)\setminus F^{\ast }(u).$

\smallskip

The paper is organized as follows. In Section 2 we prove a non-degeneracy
property which allows us to reduce Theorem \ref{flatmain2} to a normalized
form. Sections 3 and 4 are devoted to the proof of the Harnack inequality
and the improvement of flatness lemma in the non-degenerate and degenerate
case respectively. In Section 5 we exhibit a new proof of the classical
result in \cite{C1}. Section 6 deals with Perron's solutions and Theorem \ref%
{precompactenss_th}. Finally, in Section 7 we apply our results to the
Prandtl-Batchelor classical model in hydrodynamics.

\section{Non-degeneracy}

In this section we prove a non-degeneracy property which together with
compactness arguments allows us to reduce our main Theorem \ref{flatmain2}
to a normalized form.

Denote by $U_\beta$ the one-dimensional function, 
\begin{equation*}
U_\beta(t) = \alpha t^+ - \beta t^-, \quad \beta \geq 0, \quad \alpha = 
\sqrt{1 +\beta^2},
\end{equation*}
where 
\begin{equation*}
t^+ = \max\{t,0\}, \quad t^-= -\min\{t,0\}.
\end{equation*}%
Then $U_\beta(x)= U_\beta(x_n)$ is the so-called two-plane solution to %
\eqref{FBintro} when $A \equiv I$ and $f \equiv 0$.

Below, is our non-degeneracy result.

\begin{lem}
\label{deltand}Let $u$ be a solution to \eqref{FBintro} in $B_2$ with $%
Lip(u) \leq L$ and $\|f\|_{L^\infty} \leq L$. If 
\begin{equation*}
\{x_n \leq g(x^{\prime}) - \delta\} \subset \{u^+=0\} \subset \{x_n \leq
g(x^{\prime}) + \delta\},
\end{equation*}
with $g$ a Lipschitz function, $Lip(g) \leq L, g(0)=0$, then 
\begin{equation*}
u(x) \geq c_0 (x_n- g(x^{\prime})), \quad x \in \{x_n \geq g(x^{\prime}) +
2\delta\}\cap B_{\rho_0},
\end{equation*}
for some $c_0, \rho_0 >0$ depending on $n,L, \lambda, \Lambda$ as long as $%
\delta \leq c_0.$
\end{lem}

\begin{proof}
All constants in this proof will depend on $n,L, \lambda, \Lambda.$

It suffices to show that our statement holds for $\{x_n \geq g(x^{\prime}) +
C\delta\}$ for a possibly large constant $C$. Then one can apply Harnack
inequality to obtain the full statement.

We prove the statement above at $x=de_n$ (recall that $g(0)=0$). Precisely,
we want to show that 
\begin{equation*}
u(de_n) \geq c_0 d, \quad d \geq C \delta.
\end{equation*}%
After rescaling, we reduce to proving that 
\begin{equation*}
u(e_n) \geq c_0
\end{equation*}
as long as $\delta \leq 1/C$, and $\|f\|_\infty$ is sufficiently small. For $t\geq 0$, set
$$ \Omega_t:= (B_2 \setminus \bar B_1)(-te_n).$$
Solve,$$
\begin{cases} 
\mathcal L w_t = -1 \quad \text{in $\Omega_t$,}\\
w_t=0 \quad \text{on $\p B_1(-t e_n),$}\\
w_t=1 \quad \text{on $\p B_2(-t e_n)$.}
\end{cases}
$$
Extend $w_t=0$ in $B_1(-t e_n).$ By $C^{1,\bar \gamma}$ estimates and $L^\infty$ estimates (see \cite{GT}),
$$|\nabla_{A} w_t| \leq C(n,\lambda, \Lambda).$$
Hence, set $$h_t:=\frac{1}{2 C} w_t,$$ we have that 
\begin{equation*}
|\nabla _{A}h_t^+|^{2}-|\nabla _{A}h_t^-|^{2}=|\nabla _{A}h_t^+|^{2}<1,\quad %
\mbox{on}\:\:\:F(h_t).
\end{equation*}
In this way $w_t$ is a supersolution of our free boundary problem, as long as $\|f\|_\infty$ is small enough.

From our flatness assumption for $t=C(L)>0$ sufficiently large (depending on
the Lipschitz constant of $g$), $h_t$ is strictly above $u$. We decrease $t$
and let $\bar t$ be the first $t$ such that $h_t$ touches $u$ by above in $%
\Omega_{\bar t}$. Since $h_{\bar t}$ is a strict
supersolution in $\Omega_{\bar t}$ 
%and in particular in $(B_2\setminus\bar{B}_1-\bar{t}e_n)\cap \Omega^+(u),$
the touching point $z$ % in $ \Omega^+(u)$ 
can occur only on the level set $\eta=\frac{1}{2C} $ in
the positive phase of $u.$ Otherwise, as usual we would get a contradiction
with the free boundary condition by the Hopf maximum principle and the
strong maximum principle. In addition $|z|\leq C' = C'(L),$ since $u$ is
Lipschitz continuous, $0< u (z) = \eta \leq L d(z, F(u))$, that is a full
ball around $z$ of radius $\eta/L$ is contained in the positive phase of $u$%
. Thus, for $\bar \delta$ small depending on $\eta, L$ we have that $%
B_{\eta/2L}(z) \subset \{x_n \geq g(x^{\prime}) + 2 \bar \delta\}$. Since $%
x_n =g(x^{\prime}) + 2 \bar \delta$ is Lipschitz we can connect $e_n$ and $z$
with a chain of intersecting balls included in the positive side of $u$ with
radii comparable to $\eta/2L$. The number of balls depends on $L$ . Then we
can apply Harnack inequality and obtain 
\begin{equation*}
u(e_n) \geq c u(z)= c_0,
\end{equation*}
as desired.
\end{proof}

Following the argument in Section 2 \cite{DFS1}, Theorem \ref{flatmain2}
reduces to the following theorem (via a compactness argument.)

\begin{thm}
\label{main_new} Let $u$ be a solution to \eqref{FBintro} in $B_1$ with $%
Lip(u) \leq L$. There exists a universal constant $\bar \varepsilon>0$ such
that, if 
\begin{equation}  \label{initialass}
\|u - U_{\beta}\|_{L^{\infty}(B_{1})} \leq \bar \varepsilon\quad \text{for
some $0 \leq \beta \leq L,$}
\end{equation}
and 
\begin{equation*}
\{x_n \leq - \bar \varepsilon\} \subset B_1 \cap \{u^+(x)=0\} \subset \{x_n
\leq \bar \varepsilon \},
\end{equation*}
and 
\begin{equation*}
[a_{ij}]_{C^{0,\bar\gamma}(B_1)} \leq \bar \varepsilon, \quad
\|f\|_{L^\infty(B_1)} \leq \bar \varepsilon,
\end{equation*}
then $F(u)$ is $C^{1,\gamma}$ in $B_{1/2}$.
\end{thm}

The proof of Theorem \ref{main_new} follows verbatim the proof of Theorem
2.8 in \cite{DFS1}, once we have established the three key tools: the Harnack
inequality, the improvement of flatness lemma, and the dichotomy lemma.
Below we provide the proof of the dichotomy lemma, which differs from the
case $A=I$ only slightly. The Harnack inequality and the improvement of
flatness lemma are presented in the next two sections.

\begin{lem}
\label{finalcase} Let $u$ solve \eqref{FBintro} in $B_2$ with 
\begin{equation*}
\|A-I\|_{C^{0,\bar \gamma}} \leq \varepsilon^2, \quad \|f\|_{L^\infty(B_1)}
\leq \varepsilon^4
\end{equation*}
and satisfy
\begin{equation}  \label{flat**}
U_0(x_n -\varepsilon) \leq u^+(x) \leq U_0(x_n + \varepsilon) \quad \text{in 
$B_1,$} \quad 0\in F(u),
\end{equation}%
\begin{equation*}
\|u^-\|_{L^{\infty}(B_2)} \leq \bar C\varepsilon^2, \quad
\|u^-\|_{L^\infty(B_1)} > \varepsilon^2,
\end{equation*}
for a universal constant $\bar C.$ If $\varepsilon \leq \varepsilon_2$
universal, then the rescaling 
\begin{equation*}
u_\varepsilon(x) = \varepsilon^{-1/2}u(\varepsilon^{1/2}x)
\end{equation*}
satisfies in $B_1$ 
\begin{equation*}
U_{\beta^{\prime}}(x_n - C^{\prime}\varepsilon^{1/2}) \leq
u_{\varepsilon}(x) \leq U_{\beta^{\prime}}(x_n + C^{\prime}\varepsilon^{1/2})
\end{equation*}
with $\beta^{\prime}\sim \varepsilon^2$ and $C^{\prime}>0$ depending on $%
\bar C$.
\end{lem}

\begin{proof}
For notational simplicity we set 
\begin{equation*}
v = \frac{u^-}{\varepsilon^2}.
\end{equation*}
From our assumptions we can deduce that 
\begin{equation*}
F(v) \subset \{-\varepsilon \leq x_n \leq \varepsilon\},
\end{equation*}
\begin{equation}  \label{neg}
v \geq 0 \quad \text{in $B_2 \cap \{x_n \leq -\varepsilon\}$}, \quad v
\equiv 0 \quad \text{in $B_2 \cap \{x_n > \varepsilon\}.$}
\end{equation}
Also, 
\begin{equation*}
|\mathcal L v| \leq \varepsilon^2, \quad \text{in $B_2 \cap \{x_n <
-\varepsilon\}$},
\end{equation*}
and 
\begin{equation}  \label{C}
0 \leq v \leq \bar C \quad \text{on $\partial B_2$,}
\end{equation}
\begin{equation}  \label{barx}
v(\bar x) >1 \quad \text{at some point $\bar x$ in $B_1.$}
\end{equation}

Thus, using comparison with the function $\bar v$ such that 
\begin{equation*}
\mathcal L\bar v = -\varepsilon^2 \quad \text{in $D:=B_2 \cap
\{x_n < \varepsilon\}$ and $\bar v= v$ on $\partial D$}
\end{equation*}
we obtain that for some $k>0$ universal 
\begin{equation}  \label{ke}
v \leq k|x_n -\varepsilon|, \quad \text{in $B_1$}.
\end{equation}
This fact forces the point $\bar x$ in \eqref{barx} to belong to $B_1\cap
\{x_n < -\varepsilon\}$ at a fixed distance $\delta$ from $x_n =
-\varepsilon.$

Now, let $w$ be the solution to $\mathcal Lw=0$ in $B_1 \cap \{x_n < -\varepsilon\}$
such that 
\begin{equation*}
w=0 \quad \text{on $B_1 \cap \{x_n=-\varepsilon\}$}, \quad w=v \quad \text{%
on $\partial B_1 \cap \{x_n \leq -\varepsilon\}$}.
\end{equation*}
We conclude that 
\begin{equation}  \label{w-v}
|w-v| \leq c\varepsilon \quad \text{in $B_1 \cap \{x_n < -\varepsilon\}$}.
\end{equation}
In particular this is true at $\bar x$ which forces 
\begin{equation}  \label{wbarx}
w(\bar x) \geq 1/2.
\end{equation}
Furthermore, let $\tilde w$ be harmonic in $B_{9/10} \cap \{x_n < -\varepsilon\}$, with boundary data $w$. By expanding $\tilde w$ around $(0,
-\varepsilon)$ we then obtain, say in $B_{3/4} \cap \{x_n \leq
-\varepsilon\} $ 
\begin{equation*}
|\tilde w - a |x_n + \varepsilon|| \leq C |x|^2 +C\varepsilon.
\end{equation*}
Moreover, since $w$ is Lipschitz, then in $B_{3/4} \cap \{x_n \leq
-\varepsilon\}$ 
$$|w-\tilde w| \leq C \eps^2.$$
These last two inequalities combined with \eqref{w-v}  give that 
\begin{equation*}
|v - a|x_n+\varepsilon|| \leq C\varepsilon, \quad \text{in $%
B_{\varepsilon^{1/2}} \cap \{x_n \leq -\varepsilon\}$.}
\end{equation*}
In view of \eqref{wbarx} and the fact that $\bar x$ occurs at a
fixed distance from $\{x_n = -\varepsilon\}$ we deduce from Hopf lemma that 
\begin{equation*}
a \geq c>0
\end{equation*}
with $c$ universal. In conclusion (see \eqref{ke}) 
\begin{equation*}  \label{u-}
|u^- - b\varepsilon^2|x_n+\varepsilon|| \leq C\varepsilon^3, \quad \text{in $%
B_{\varepsilon^{1/2}} \cap \{x_n \leq -\varepsilon\}$,} \quad u^- \leq
b\varepsilon^2 |x_n-\varepsilon|, \quad \text{in $B_1$}
\end{equation*}
with $b$ comparable to a universal constant.

Combining the two inequalities above and the assumption \eqref{flat**} we
conclude that in $B_{\varepsilon^{1/2}}$ 
\begin{equation*}
(x_n - \varepsilon)^+ - b\varepsilon^2(x_n-C\varepsilon)^- \leq u(x) \leq
(x_n+\varepsilon)^+ - b \varepsilon^2 (x_n+C\varepsilon)^-
\end{equation*}
with $C>0$ universal and $b$ larger than a universal constant. Rescaling, we
obtain that in $B_1$ 
\begin{equation*}
(x_n - \varepsilon^{1/2})^+ - \beta^{\prime}(x_n-C\varepsilon^{1/2})^- \leq
u_\varepsilon(x) \leq (x_n+\varepsilon^{1/2})^+ -
\beta^{\prime}(x_n+C\varepsilon^{1/2})^-
\end{equation*}
with $\beta^{\prime}\sim \varepsilon^2$. We finally need to check that this
implies the desired conclusion in $B_1:$ 
\begin{equation*}
\alpha^{\prime}(x_n - C\varepsilon^{1/2})^+ -
\beta^{\prime}(x_n-C\varepsilon^{1/2})^- \leq u_\varepsilon(x) \leq
\alpha^{\prime}(x_n+C\varepsilon^{1/2})^+ -
\beta^{\prime}(x_n+C\varepsilon^{1/2})^-
\end{equation*}
with $(\alpha^{\prime})^2=1+(\beta^{\prime})^2 \sim 1+ \varepsilon^4.$ This
clearly holds in $B_1$ for $\varepsilon$ small, say by possibly enlarging $C$
so that $C \geq 2.$
\end{proof}

\section{Non-degenerate case}

In this section we prove the Harnack inequality and the improvement of
flatness lemma in the so-called non-degenerate case. In this case our
solution $u$ is trapped between two translations of a ``true" two-plane
solution $U_\beta$ that is ($L=Lip(u)$)
\begin{equation*}
0 < \beta \leq L.
\end{equation*}

\subsection{Harnack inequality}

We start with the Harnack Inequality.

\begin{thm}[Harnack inequality]
\label{HI} There exists a universal constant $\bar \varepsilon$, such that
if $u$ satisfies at some point $x_0 \in B_2$

\begin{equation}  \label{osc}
U_\beta(x_n+ a_0) \leq u(x) \leq U_\beta(x_n+ b_0) \quad \text{in $B_r(x_0)
\subset B_2,$}
\end{equation}
with 
\begin{equation}  \label{G}
\|f\|_{L^\infty(B_1)} \leq \varepsilon^2 \beta, \quad \|A-I\|_{C^{0,\bar
\gamma}}\leq \varepsilon^2,
\end{equation}
and 
\begin{equation*}
b_0 - a_0 \leq \varepsilon r,
\end{equation*}
for some $\varepsilon \leq \bar \varepsilon,$ then 
\begin{equation*}
U_{\beta}(x_n+ a_1) \leq u(x) \leq U_\beta(x_n+ b_1) \quad \text{in $%
B_{r/20}(x_0)$},
\end{equation*}
with 
\begin{equation*}
a_0 \leq a_1 \leq b_1 \leq b_0, \quad b_1 - a_1\leq (1-c)\varepsilon r,
\end{equation*}
and $0<c<1$ universal.
\end{thm}

We deduce the following Corollary for the oscillation of $u$, i.e., 
\begin{equation*}
\tilde u_\varepsilon(x) = 
\begin{cases}
\dfrac{u(x) -\alpha x_n }{\alpha\varepsilon} \quad \text{in $B_2^+(u) \cup
F(u)$} \\ 
\  \\ 
\dfrac{u(x) -\beta x_n }{\beta\varepsilon} \quad \text{in $B_2^-(u).$}%
\end{cases}%
\end{equation*}

\begin{cor}
\label{corollary}Let $u$ be as in Theorem $\ref{HI}$ satisfying \eqref{osc}
for $r=1$. Then in $B_1(x_0)$ $\tilde u_\varepsilon$ has a H\"older modulus
of continuity at $x_0$, outside the ball of radius $\varepsilon/\bar
\varepsilon,$ i.e for all $x \in B_1(x_0)$, with $|x-x_0| \geq
\varepsilon/\bar\varepsilon$ 
\begin{equation*}
|\tilde u_\varepsilon(x) - \tilde u_\varepsilon (x_0)| \leq C |x-x_0|^\gamma.
\end{equation*}
\end{cor}

The proof of the Harnack inequality follows from the next lemma (see \cite%
{DFS1} for details).

\begin{lem}
\label{main}There exists a universal constant $\bar \varepsilon>0$ such that
if \eqref{G} holds for $0<\varepsilon < \bar \varepsilon$, and $u$ satisfies 
\begin{equation*}
u(x) \geq U_\beta(x), \quad \text{in $B_1$}
\end{equation*}
and at $\bar x=\dfrac{1}{5}e_n$ 
\begin{equation}  \label{u-p>ep2}
u(\bar x) \geq U_\beta(\bar x_n + \varepsilon),
\end{equation}
then 
\begin{equation}
u(x) \geq U_\beta(x_n+c\varepsilon), \quad \text{in $\overline{B}_{1/2},$}
\end{equation}
for some $0<c<1$ universal. Analogously, if 
\begin{equation*}
u(x) \leq U_\beta(x), \quad \text{in $B_1$}
\end{equation*}
and 
\begin{equation*}
u(\bar x) \leq U_\beta(\bar x_n - \varepsilon),
\end{equation*}
then 
\begin{equation*}
u(x) \leq U_\beta(x_n - c \varepsilon), \quad \text{in $\overline{B}_{1/2}.$}
\end{equation*}
\end{lem}

\begin{proof}
We prove the first statement. The proof of the second statement is similar.
For notational simplicity we drop the sub-index $\beta$ from $U_\beta.$

Since $x_n >0$ in $B_{1/10}(\bar x)$ and $u \geq U$ in $B_1$ we get 
\begin{equation*}
B_{1/10}(\bar x) \subset B_1^+(u).
\end{equation*}

Thus $u-U \geq 0$ and solves $\mathcal{L }(u-U) =f - \mathcal{L}U$ in $%
B_{1/10}(\bar x)$. By Harnack inequality we obtain 
\begin{equation}  \label{HInew}
u(x) - U(x) \geq c(u(\bar x)- U(\bar x)) - C (\|f\|_{L^\infty} +
\|(A-I)\nabla U\|_{C^{0,\bar \gamma}}) \quad \text{in $\overline
B_{1/20}(\bar x)$}.
\end{equation}
From the assumptions \eqref{G} and \eqref{u-p>ep2} we conclude that (for $%
\varepsilon$ small enough) 
\begin{equation}  \label{u-p>cep}
u - U \geq \alpha c\varepsilon - C \alpha \varepsilon^2 \geq \alpha
c_0\varepsilon \quad \text{in $\overline B_{1/20}(\bar x)$}.
\end{equation}
This is the desired statement \eqref{u-p>ep2} in the ball $B_{1/20}(\bar x).$
We now work in annuli 
\begin{equation*}
D_r:= B_r(\bar x) \setminus \overline{B}_{1/20}(\bar x).
\end{equation*}

Let 
\begin{equation}
w=c(|x-\bar{x}|^{-\eta }-(3/4)^{-\eta }),\quad x\in \overline{D}_{7/8},
\label{w}
\end{equation}%
with $c$ such that 
\begin{equation*}
w=1\quad \text{on $\partial B_{1/20}(\bar{x})$,}
\end{equation*}%
and $\eta $ fixed larger than $n-2,$ so that 
\begin{equation}
\Delta w\geq k(\eta ,n)=k(n)>0,\quad \text{on $D_{7/8}$.}  \label{lapw}
\end{equation}
Notice that $w$ becomes negative outside $B_{3/4}(\bar x).$ %
%
%Extend $w$ to be equal
%to 1 on $B_{1/20}(\bar x).$

Set $\bar w=1-w$ and for $t \geq 0,$ 
\begin{equation*}
v_t(x)= U(x_n - \varepsilon c_0 \bar w(x)+t\varepsilon), \quad x \in
\overline D_{7/8}.
\end{equation*}

Then,

\begin{equation*}
v_0(x)=U(x_n - \varepsilon c_0 \bar w(x)) \leq U(x) \leq u(x) \quad x \in
\overline D_{7/8}.
\end{equation*}

In particular 
\begin{equation}  \label{inclusion}
\{ v_{0} >0\} \subset \{u>0\}.
\end{equation}

Let $\bar t$ be the largest $t \geq 0$ such that 
\begin{equation}  \label{inclusion2}
\{ v_{t} >0\} \subset \{u>0\}.
\end{equation}

%Call $D^+_t= \{v_{t} >0\} \cap D_{7/8}$ and 
%let $\bar t$ be the largest $t \geq 0$ such that
%$$v_{t}(x) \leq u(x) \quad \text{on $\p D^+_t$}.$$
%Thus at some $\tilde x \in \p D^+_{\bar t}$ we have $$v_{\bar t}(\tilde x) =u(\tilde x).$$

%{\bf Claim.} $\bar t \geq c_0 \eps.$

%Assume by contradiction that $\bar t < c_0 \eps$. Then the first touching point $\tilde x$ must occur on $F(v_{\bar t}) \cap D_{3/4}.$

For any $t \leq \bar t$, call 
\begin{equation*}
D^+_{r,t}= \{v_{t} >0\} \cap D_{r}, D^-_{r,t}= \{v_{t} <0\} \cap D_{r}.
\end{equation*}
Let $\phi_{t}^\pm$ be the solution to 
\begin{equation}  \label{phit}
\begin{cases}
\mathcal{L }\phi_{t}^\pm = - \text{div}((A-I)\nabla v_{t})\quad \text{in $%
D^\pm_{7/8,t}$} \\ 
\phi_{t}^\pm=0 \quad \text{on $\partial D^\pm_{7/8,t}.$}%
\end{cases}%
\end{equation}

By assumption \eqref{G} and the boundary regularity estimates for divergence
form equations (see \cite{GT}) we get 
\begin{equation}  \label{estphi}
\|\phi_{t}^+\|_{C^{1,\bar \gamma}} \leq C\alpha \varepsilon^2, \quad \text{%
in $\overline{D^+_{6/7,t }}$} \quad \|\phi_{t}^-\|_{C^{1,\bar \gamma}} \leq
C\beta \varepsilon^2, \quad \text{in $\overline{D^-_{6/7,t }}$}
\end{equation}
and by $L^\infty$ estimates (see \cite{GT}) 
\begin{equation}  \label{linf}
\|\phi_{t}^+\|_{L^\infty} \leq C \alpha\varepsilon^2, \quad \text{in $%
\overline{D_{7/8,t}^+},$} \quad \|\phi_{t}^-\|_{L^\infty} \leq C
\beta\varepsilon^2, \quad \text{in $\overline{D_{7/8,t}^-}.$}
\end{equation}
Call 
\begin{equation*}
\phi_{t} = 
\begin{cases}
\phi_{t}^+ \text{in $\overline{D^+_{7/8,t}}$} \\ 
\phi^-_{t} \text{in $\overline{D^-_{7/8,t}}$} \\ 
\end{cases}%
\end{equation*}
Now set, 
\begin{equation*}
\psi_{t}= v_{t} + \phi_{t}.
\end{equation*}%
Since $w_n$ is bounded in the annulus $D_{7/8}$, we easily obtain that for $%
\varepsilon$ small enough

\begin{equation*}
(v_{t})_n \geq \alpha/2 \quad \text{in $D^+_{7/8,t }$} \quad (v_{t})_n \geq
\beta/2 \quad \text{in $D^-_{7/8,t }$}.
\end{equation*}

Thus, from the estimates above we conclude that 
\begin{equation*}
(\psi_{t})_n >0 \quad \text{in $\overline{D^+}_{6/7,t }\cup \overline{D^-}%
_{6/7,t}$}.
\end{equation*}
Hence, since $F(\psi_t)$ is a graph in the $e_n$ direction ($\varepsilon$
small), we get that 
\begin{equation*}
\{\psi_{t} >0\} \cap D_{6/7} = \{v_{t} >0\} \cap D_{6/7} = D^+_{6/7,t},
\end{equation*}
\begin{equation}  \label{2}
\{\psi_{t} <0\} \cap D_{6/7} = \{v_{t} <0\} \cap D_{6/7}=D^-_{6/7,t} ,
\end{equation}
\begin{equation}  \label{3}
F(\psi_{t}) \cap D_{6/7}= F(v_{ t}) \cap D_{6/7}.
\end{equation}
Moreover, in view of \eqref{lapw}$, \psi_{ t}$ solves 
\begin{equation}  \label{top}
\mathcal{L }\psi_{ t} = \Delta v_{ t}\geq \varepsilon \beta c_0 k(n)\quad 
\text{in $D^+_{7/8,t} \cup D^-_{7/8,t}$}.
\end{equation}
From \eqref{inclusion2}, \eqref{top} and assumption \eqref{G} we have

\begin{equation*}
\mathcal{L }\psi_{t} > \mathcal{L }u \quad \text{in $D^+_{7/8,t}.$}
\end{equation*}

If $t \leq \min\{c_0, \bar t\}$, using \eqref{u-p>cep} and the fact that $u
\geq U$ we obtain that

\begin{equation*}
\psi_{t}= v_{ t} \leq u \quad \text{on $\partial D^+_{7/8,t}.$}
\end{equation*}
Thus, by the maximum principle 
\begin{equation*}
\psi_{t} \leq u \quad \text{in $D^+_{7/8,t}.$}
\end{equation*}
In particular, 
\begin{equation}  \label{up}
\mathcal{L }\psi_{t} > f, \quad \psi_{t} \leq u \quad \text{in $\{\psi_{t}
>0\} \cap D_{6/7}.$}
\end{equation}

Analogously, 
\begin{equation*}
\mathcal{L }\psi_{t} > \mathcal{L }u \quad \text{in $\{u<0\} \cap D_{7/8,t}^-,$}
\end{equation*}
and using \eqref{2} 
\begin{equation*}
\psi_{t} \leq u \quad \text{on $F(u) \cap D^-_{6/7,t}$}.
\end{equation*}
Moreover, since $u \geq U$, using the definition of $v_{t}$ and the fact
that $w < - \tilde c < 0$ outside $B_{6/7}(\bar x)$, we have that 
\begin{equation}\label{help}
%v_ {t} - u \leq -c \alpha \varepsilon, \quad \text{in $\overline{D^+_{7/8,t}}
%\setminus D_{6/7}$}, \quad 
v_ {t} - u \leq -c \beta \varepsilon, \quad \text{%
in $(\overline{D^-_{7/8,t}} \setminus D_{6/7})$}.
\end{equation}
Thus, from the $L^\infty$ estimates \eqref{linf}, we obtain that (for $%
\varepsilon$ small enough)

\begin{equation*}
\psi_{t} <u \quad \text{in $\overline{D_{7/8,t}^-} \setminus D_{6/7}$}.
\end{equation*}

Now, by the maximum principle,

\begin{equation*}
\psi_{ t} \leq u, \quad \text{in $\{u<0\} \cap D^-_{6/7,t}.$}
\end{equation*}
In particular,

\begin{equation}  \label{down}
\mathcal{L }\psi_{t} > f,\quad \psi_{t} \leq u, \quad \text{in $\{\psi_{t}
<0\} \cap D_{6/7}.$}
\end{equation}

On the other hand, $\psi_{t}$ satisfies 
\begin{equation*}
|\nabla _A \psi_{ t}^+|^2 - |\nabla_A \psi_{ t}^-|^2 > 1 \quad \text{on $%
F(\psi_{t}) \cap D_{6/7}$}.
\end{equation*}
Indeed, by assumption \eqref{G} 
\begin{equation*}
|\nabla _A \psi_{ t}^+|^2 > (1-\varepsilon^2)|\nabla \psi_t^+|^2
\end{equation*}
and 
\begin{equation*}
|\nabla _A \psi_{ t}^-|^2 \leq (1+\varepsilon^2)|\nabla \psi_t^-|^2.
\end{equation*}
Furthermore, by the $C^{1,\bar \gamma}$ estimate \eqref{estphi} and
assumption \eqref{G}, we have that on $F(\psi_{t}) \cap D_{6/7}$, (recall $%
\alpha, \beta$ are bounded) 
\begin{equation*}
|\nabla \psi_{ t}^+|^2 \geq(1-\varepsilon^2) |\nabla v_{t}^+|^2 -
C^{\prime}\varepsilon^2, \quad |\nabla \psi_{ t}^-|^2 \leq
(1+\varepsilon^2)|\nabla v_{ t}^-|^2 + C^{\prime\prime}\varepsilon^2,
\end{equation*}
with, 
\begin{equation*}
|\nabla v_{t}^+|^2 = \alpha^2(1 + \varepsilon^2 c_0^2 |\nabla w|^2 +
2\varepsilon c_0 w_n), \quad |\nabla v_{t}^-|^2 = \beta^2(1 + \varepsilon^2
c_0^2 |\nabla w|^2 + 2\varepsilon c_0 w_n).
\end{equation*}

Hence, for $\bar C$ universal, 
\begin{equation*}
|\nabla _A \psi_{ t}^+|^2 - |\nabla_A \psi_{ t}^-|^2 \geq (1 + 2\varepsilon
c_0 w_n) - \bar C \varepsilon^2 \quad \text{on $F(\psi_{ t}) \cap D_{6/7}$}.
\end{equation*}

Using that $w_n \geq c>0$ on $F(\psi_{t}) \cap D_{6/7}$ we obtain the
desired claim by choosing $\varepsilon$ small enough.

In conclusion, $\psi_{t}$ is a strict subsolution in $D_{6/7}$ which lies
below $u$, if $t \leq \min\{\bar t, c_0\}$.

However, if $\bar t \leq c_0$, then $F(\psi_{\bar t})$ touches $F(u)$ in $%
D_{6/7}$, which is a contradiction. Indeed, for $t \leq c_0 $, $v_t < u$ on $%
(\overline{B_{7/8}(\bar x)} \setminus B_{6/7}(\bar x)) \cup \p B_{1/20}(\bar x)$ (see \eqref{u-p>cep}, \eqref{help} and its analogue in $\overline{D^+_{7/8}} \setminus D_{6/7}.$)  Since $\bar t$ is the largest $t \geq 0$ for which the
inclusion \eqref{inclusion2} holds, we deduce that $F(v_{\bar t})
\cap F(u) \neq \emptyset$ in $D_{6/7}$. It follows from \eqref{3} that $%
F(\psi_{\bar t}) \cap F(u) \neq \emptyset$ in $D_{6/7},$ as desired.

Thus $\bar t > c_0.$ In particular, for $t=c_0$, 
\begin{equation*}
\psi_{c_0} \leq u \quad \text{in $D_{6/7}.$}
\end{equation*}

Then we get the desired statement. Indeed, by the $L^\infty$ estimate we
obtain in $(B_{1/2} \cap \{\psi_{c_0}>0\}) \setminus B_{1/20}(\bar x)
\subset \subset D_{6/7}$%
\begin{equation*}
u(x) \geq \psi_{c_0}(x) \geq -C\alpha\varepsilon^2 + U(x_n + c_0 \varepsilon
w) \geq U(x_n + c\varepsilon)
\end{equation*}
and analogously, in $(B_{1/2} \cap \{\psi_{c_0}\leq 0\}) \setminus
B_{1/20}(\bar x) \subset \subset D_{6/7}$%
\begin{equation*}
u(x) \geq \psi_{c_0}(x) \geq -C\beta \varepsilon^2 + U(x_n + c_0 \varepsilon
w) \geq U(x_n + c\varepsilon)
\end{equation*}%
with $c$ universal. Above we used that $w \geq \bar c >0$ on $B_{1/2},$ and $%
\varepsilon$ is chosen sufficiently small.
\end{proof}

%We can now prove our Theorem \ref{HI}.

%\vspace{3mm}

%\textit{Proof of Theorem $\ref{HI}$.} Assume without loss of generality that 
%$x_0=0, r=1.$ We distinguish three cases.

%\vspace{1mm}

%\textit{Case 1.} $a_0 < - 1/5.$ In this case it follows from \eqref{osc}
%that $B_{1/10} \subset \{u<0\}$ and 
%\begin{equation*}
%0 \leq v(x):= \frac{u(x) - \beta(x_n+a_0)}{\beta \varepsilon} \leq 1
%\end{equation*}
%with 
%\begin{equation*}
%\mathcal{L }v = \frac{1}{\varepsilon} (f - \text{div}((A-I)e_n)) \quad \text{%
%in $B_{1/10}$.}
%\end{equation*}
%The desired claim follows from standard Harnack inequality applied to $v$.

%\vspace{1mm}

%\textit{Case 2.} $a_0 > 1/5.$ We argue as above noticing that $B_{1/5}
%\subset \{u>0\}$ and choosing 
%\begin{equation*}
%0 \leq v(x):= \frac{u(x) - \alpha(x_n+a_0)}{\alpha \varepsilon} \leq 1.
%\end{equation*}%
%\vspace{1mm}

%\textit{Case 3.} $|a_0| \leq 1/5.$ This case follows immediately from Lemma %
%\ref{main} (see ...)

\subsection{Improvement of flatness}

We are now ready to prove the improvement of flatness lemma in the
non-degenerate case.

\begin{lem}[Improvement of flatness]
\label{improv1}Let $u$ satisfy 
\begin{equation}  \label{flat_1}
U_\beta(x_n -\varepsilon) \leq u(x) \leq U_\beta(x_n + \varepsilon) \quad 
\text{in $B_1,$} \quad 0\in F(u),
\end{equation}
with 
\begin{equation}  \label{ass1-ass2}
\|A-I\|_{C^{0,\bar \gamma}} \leq \varepsilon^2, \quad \|f\|_{L^\infty(B_1)}
\leq \varepsilon^2\beta.
\end{equation}

If $0<r \leq r_0$ for $r_0$ universal, and $0<\varepsilon \leq \varepsilon_0$
for some $\varepsilon_0$ depending on $r$, then 
%\begin{equation}\label{improvedflat_1}(x \cdot \nu -r\frac{\ep^2}{2} -C r^2\ep)^+ \leq u(x) \leq
%(x \cdot \nu +r\frac{\ep^2}{2} +C r^2\ep)^+ \quad \text{for $x \in
%B_r$}
%\end{equation}

%In particular, if $C r \leq 1/4$ and $\ep_0 \leq 1/2$ then
\begin{equation}  \label{improvedflat_2_new}
U_{\beta^{\prime }}(x \cdot \nu_1 -r\frac{\varepsilon}{2}) \leq u(x) \leq
U_{\beta^{\prime }}(x \cdot \nu_1 +r\frac{\varepsilon }{2}) \quad \text{in $%
B_r,$}
\end{equation}
with $|\nu_1|=1,$ $|\nu_1 - e_n| \leq \tilde C\varepsilon$ , and $|\beta
-\beta^{\prime }| \leq \tilde C\beta \varepsilon$ for a universal constant $%
\tilde C.$
\end{lem}

\begin{proof}
The proof of this Lemma is divided into 3 steps.

\vspace{2mm}

\textbf{Step 1 -- Compactness.} Fix $r \leq r_0$ with $r_0$ universal (made
precise in Step 3). Assume by contradiction that we can find a sequence $%
\varepsilon_k \rightarrow 0$ and a sequence $u_k$ of solutions to 
\begin{equation*}
\begin{cases}
\mathcal{L}_k u_k := \text{div}(A_k\nabla u_k) = f_k \quad \text{in $%
B_1^+(u_k) \cup B_1^-(u_k)$} \\ 
|\nabla_{A_k} u_k^+|^2 - |\nabla_{A_k} u_k^-|^2=1 \quad \text{on $F(u_k)$}%
\end{cases}%
\end{equation*}
with 
\begin{equation}  \label{ak}
\|f_k\|_{L^\infty} \leq \beta_k\varepsilon_k^2, \quad \|A_k -I\|_{C^{0,\bar
\gamma}} \leq \varepsilon_k^2
\end{equation}
such that 
\begin{equation}  \label{flat_k}
U_{\beta_k}(x_n -\varepsilon_k) \leq u_k(x) \leq U_{\beta_k}(x_n +
\varepsilon_k) \quad \text{for $x \in B_1$, $0 \in F(u_k),$}
\end{equation}
with $L \geq \beta_k > 0$, but $u_k$ does not satisfy the conclusion %
\eqref{improvedflat_2_new} of the lemma.

Set ($\alpha_k^2=1+\beta_k^2$),%
\begin{equation*}
\tilde{u}_{k}(x)= 
\begin{cases}
\dfrac{u_k(x) - \alpha_k x_n}{\alpha_k \varepsilon_k}, \quad x \in
B_1^+(u_k) \cup F(u_k) \\ 
\  \\ 
\dfrac{u_k(x) - \beta_k x_n}{\beta_k\varepsilon_k}, \quad x \in B_1^-(u_k).%
\end{cases}%
\end{equation*}
Then \eqref{flat_k} gives, 
\begin{equation}  \label{flat_tilde**}
-1 \leq \tilde{u}_{k}(x) \leq 1 \quad \text{for $x \in B_1$}.
\end{equation}

From Corollary \ref{corollary}, it follows that the function $\tilde u_{k}$
satisfies 
\begin{equation}  \label{HC}
|\tilde u_{k}(x) - \tilde u_{k} (y)| \leq C |x-y|^\gamma,
\end{equation}
for $C$ universal and 
\begin{equation*}
|x-y| \geq \varepsilon_k/\bar\varepsilon, \quad x,y \in B_{1/2}.
\end{equation*}
From \eqref{flat_k} it clearly follows that $F(u_k)$ converges to $B_1 \cap
\{x_n=0\}$ in the Hausdorff distance. This fact and \eqref{HC} together with
Ascoli-Arzela give that as $\varepsilon_k \rightarrow 0$ the graphs of the $%
\tilde{u}_{k}$ converge (up to a subsequence) in the Hausdorff distance to
the graph of a H\"older continuous function $\tilde{u}$ over $B_{1/2}$.
Also, up to a subsequence 
\begin{equation*}
\beta_k \to \tilde \beta \geq 0
\end{equation*}
and hence 
\begin{equation*}
\alpha_k \to \tilde \alpha = \sqrt{1+\tilde \beta^2}.
\end{equation*}

\vspace{2mm}

\textbf{Step 2 -- Limiting Solution.} We show that $\tilde u$ solves the
linearized problem 
\begin{equation}  \label{Neumann}
\begin{cases}
\Delta \tilde u=0 & \text{in $B_{1/2} \cap \{x_n \neq 0\}$}, \\ 
\  &  \\ 
\tilde \alpha^2 (\tilde u_n)^+ - \tilde\beta^2 (\tilde u_n)^-=0 & \text{on $%
B_{1/2} \cap \{x_n =0\}$}.%
\end{cases}%
\end{equation}

Since 
\begin{equation*}
\Delta u_{k} = f_k - \text{div}((A_k-I)\nabla u_k) \quad \text{in $%
B_1^+(u_k) \cup B^-_1(u_k)$},
\end{equation*}
one easily deduces that $\tilde u$ is harmonic in $B_{1/2} \cap \{x_n \neq
0\}$.

Next, we prove that $\tilde u$ satisfies the boundary condition in %
\eqref{Neumann} in the viscosity sense.

Assume by contradiction that there exists a function $\tilde \phi$ of the
form 
\begin{equation*}
\tilde \phi(x) = M+ px_n^+- qx_n^- + B Q(x-y)
\end{equation*}
with 
\begin{equation*}
Q(x) = \frac 1 2 [(n-1)x_n^2 - |x^{\prime}|^2], \quad y=(y^{\prime },0),
\quad M \in \mathbb{R}, B >0
\end{equation*}
and 
\begin{equation*}
\tilde \alpha^2 p- \tilde \beta^2 q>0,
\end{equation*}
which touches $\tilde u$ strictly by below at a point $x_0= (x_0^{\prime },
0) \in B_{1/2}$.

As in \cite{DFS1}, let 
\begin{equation}  \label{biggamma}
\Gamma(x) = \frac{1}{n-2} [(|x^{\prime}|^2+ |x_n-1|^2)^{\frac{2-n}{2}} - 1]
\end{equation}
and let

\begin{equation}  \label{biggammak}
\Gamma_{k}(x)= \frac{1}{B\varepsilon_k}\Gamma(B\varepsilon_k(x-y)+ M B
\varepsilon_k^2 e_n).
\end{equation}

Now, call 
\begin{equation*}
\phi_k(x)= a_k \Gamma^+_k(x) - b_k \Gamma^-_k(x) + \alpha_k
(d_k^+(x))^2\varepsilon_k^{3/2} +\beta_k(d_k^-(x))^2\varepsilon_k^{3/2}
\end{equation*}
where 
\begin{equation*}
a_k=\alpha_k(1+\varepsilon_k p), \quad b_k=\beta_k(1+\varepsilon_k q)
\end{equation*}
and $d_k(x)$ is the signed distance from $x$ to $\partial B_{\frac{1}{%
B\varepsilon_k}}(y+e_n(\frac{1}{B\varepsilon_k}-M\varepsilon_k)).$

%By Taylor's theorem
%$$ \Gamma(x)= x_n + Q(x) + O(|x|^3) \quad x \in
%B_1,$$
%thus it is easy to verify that
%$$ \Gamma_k(x)= A\ep_k + x_n + B\ep_kQ(x-y) + O(\ep_k^2) \quad x \in
%B_1,$$
%with the constant in $O(\ep_k^2)$ depending on $A,B,$ and $|y|$ (later this constant will depend also on $p,q.$).

%It follows that in $
%B_1^+(\phi_k) \cup F(\phi_k) $ ($Q^y(x)=Q(x-y)$)
%$$ \tilde{\phi}_{k}(x)= A+BQ^y + p x_n + A\ep_k p + Bp\ep_kQ^y + \ep_k^{1/2}d_k^2+ O(\ep_k)
%$$
%and analogously in  $
%B_1^-(\phi_k)$

%$$ \tilde{\phi}_{k}(x)= A+BQ^y + q x_n + A\ep_k p + Bq\ep_kQ^y + \ep_k^{1/2}d^2_k + O(\ep_k).
%$$

Now, call $g_k^\pm$ the solution to

\begin{equation*}
\begin{cases}
\mathcal{L}_k g_k^\pm = -\text{div}((A_k-I)\nabla \phi_k) \quad \text{in $%
B_1^\pm(\phi_k)$} \\ 
g_k=0 \quad \text{on $F(\phi_k).$}%
\end{cases}%
\end{equation*}

By $C^{1,\bar \gamma}$ and $L^\infty$ estimates, using the first one of %
\eqref{ass1-ass2} and the formula for $\phi_k$, we get
\begin{equation*}
\|g_k^+\|_{C^{1,\bar \gamma}} \leq C(a_k+\alpha_k) \varepsilon_k^2 \quad 
\text{in $\overline{B}_{7/8}^+(\phi_k),$} \quad \|g_k^-\|_{C^{1,\bar
\gamma}} \leq C(b_k+\beta_k) \varepsilon_k^2 \quad \text{in $\overline{B}%
_{7/8}^-(\phi_k)$,}
\end{equation*}
and 
\begin{equation*}
\|g_k^+\|_{L^\infty} \leq C(a_k+\alpha_k)\varepsilon_k^2, \quad \text{in $%
\overline{B_1}^+(\phi_k),$} \quad \|g_k^-\|_{L^\infty} \leq
C(b_k+\beta_k)\varepsilon_k^2, \quad \text{in $\overline{B_1}^-(\phi_k).$}
\end{equation*}

Set,

\begin{equation*}
\gamma_k^\pm := \phi_k^\pm + g_k^\pm,
\end{equation*}

and

\begin{equation*}
\tilde{\gamma}_{k}(x)= 
\begin{cases}
\tilde{\phi}_k + \dfrac{g_k^+}{\alpha_k \varepsilon_k}= \dfrac{\phi_k(x) -
\alpha_k x_n}{\alpha_k \varepsilon_k}+ \dfrac{g_k^+}{\alpha_k \varepsilon_k}%
, \quad x \in B_1^+(\phi_k) \cup F(\phi_k) \\ 
\  \\ 
\tilde{\phi}_k + \dfrac{g_k^-}{\beta_k\varepsilon_k}= \dfrac{\phi_k(x) -
\beta_k x_n}{\beta_k \varepsilon_k}+ \dfrac{g_k^-}{\beta_k\varepsilon_k}%
\quad x \in B_1^-(\phi_k).%
\end{cases}%
\end{equation*}
As shown in \cite{DFS1}, $\tilde \phi_k$ converges uniformly to $\tilde \phi$
on $B_{1/2}$. Thus, from the $L^\infty$ estimates, also $\tilde{\gamma}_k$
converges uniformly to $\tilde \phi$ on $B_{1/2}.$

Since $\tilde u_k$ converges uniformly to $\tilde u$ and $\tilde \phi$
touches $\tilde u$ strictly by below at $x_0$, we conclude that there exist
a sequence of constants $c_k \to 0$ and of points $x_k \to x_0$ such that
the function 
\begin{equation*}
\psi_k(x) = \gamma_k(\bar x), \quad \bar x = x+\varepsilon_k c_k e_n
\end{equation*}
touches $u_k$ by below at $x_k$. We will get a contradiction by proving that 
$\psi_k$ is a strict subsolution to our free boundary problem in $B_{7/8}$,
that is 
\begin{equation*}  \label{fbpsi*}
\left \{ 
\begin{array}{ll}
\mathcal{L}_k \psi_k > \varepsilon_k^2\beta_k \geq \|f_k\|_{\infty}, & %
\hbox{in $B_{7/8}^+(\psi_k) \cup B_{7/8}^-(\psi_k),$} \\ 
\  &  \\ 
|\nabla_{A_k} \psi_k^+|^2 - |\nabla_{A_k} \psi_k^-|^2 >1, & 
\hbox{on
$F(\psi_k)$.} \\ 
& 
\end{array}%
\right.
\end{equation*}

First, we show that

\begin{equation}  \label{equality}
B_{7/8}^\pm(\gamma_k) = B_{7/8}^\pm(\phi_k), \quad F(\gamma_k) \cap B_{7/8}
= F(\phi_k) \cap B_{7/8}.
\end{equation}

To this aim, notice that for $k$ large enough $\partial_n\Gamma^\pm_k \geq c
>0$ and $\partial_n (d_k^\pm)^2$ is bounded, thus 
\begin{equation*}
\partial_n \phi_k \geq ca_k -C \alpha_k \varepsilon_k^{3/2} \quad \text{in $%
B_1^+(\phi_k),$}
\end{equation*}
\begin{equation*}
\partial_n \phi_k \geq cb_k -C \beta_k \varepsilon_k^{3/2}\quad \text{in $%
B_1^-(\phi_k)$.}
\end{equation*}
These inequalities combined with the $C^{1,\bar \gamma}$ estimates for $g_k$
give that 
\begin{equation*}
\partial_n \gamma_k^\pm >0 \quad \text{in $\overline{B}_{7/8}^\pm(\phi_k)$ }.
\end{equation*}
Since $F(\phi_k)$ is a graph in the $e_n$ direction (for $k$ large), we
deduce the desired claim \eqref{equality}.

It is easily checked that away from the free boundary 
\begin{equation*}
\mathcal{L}_k \psi_k = \Delta \phi_k(\bar x) \geq \beta_k
\varepsilon_k^{3/2} \Delta d_k^2(\bar x)
\end{equation*}
and the first condition is satisfied for $k$ large enough.

Finally, since on the zero level set $|\nabla \Gamma_k|=1$ and $|\nabla
d^2_k|=0$ we have that

\begin{equation*}
(\phi_k^+)_\nu^2(\bar x) = a_k^2, \quad(\phi_k^-)_\nu^2(\bar x)= b_k^2.
\end{equation*}

By the second equation in \eqref{ak}, 
\begin{equation*}
|\nabla_{A_k} \psi_k^+|^2 - |\nabla_{A_k} \psi_k^-|^2 \geq
(1-\varepsilon_k)|\nabla \psi_k^+|^2 - (1+\varepsilon_k^2)|\nabla
\psi_k^-|^2.
\end{equation*}

Moreover, using the $C^{1,\bar \gamma}$ estimates for $g_k^\pm$ we get that

\begin{equation*}
|\nabla \psi_k^+|^2 \geq (1-\varepsilon_k^2)a_k^2 -
C^{\prime}\varepsilon_k^2, \quad |\nabla \psi_k^-|^2 \leq
(1+\varepsilon_k^2)b_k^2 +C^{\prime\prime}\varepsilon_k^2.
\end{equation*}

These, combined with the estimate above and the definition of $a_k$ and $b_k$
give that 
\begin{equation*}
|\nabla_{A_k} \psi_k^+|^2 - |\nabla_{A_k} \psi_k^-|^2 \geq 1+ (\alpha_k^2
p^2 - \beta_k^2 q^2)\varepsilon_k^2 +
2\varepsilon_k(\alpha_kp-\beta_kq)-C\varepsilon_k^2>1.
\end{equation*}
The last inequality holds for $k$ large in view of the fact that 
\begin{equation*}
\tilde \alpha^2 p -\tilde \beta^2 q >0.
\end{equation*}

\vspace{2mm}

\textbf{Step 3 -- Contradiction.} This step follows as in \cite{DFS1}.
\end{proof}

\section{Degenerate case}

In this section we prove the Harnack inequality and the improvement of
flatness lemma, in the so-called degenerate case. In this case, the negative
part of $u$ is negligible and the positive part is close to a one-plane
solution (i.e. $\beta=0$).

\subsection{Harnack inequality.}

We start with the Harnack inequality.

\begin{thm}[Harnack inequality]
\label{HIdg}There exists a universal constant $\bar \varepsilon$, such that
if $u$ satisfies at some point $x_0 \in B_2$

\begin{equation}
U_{0}(x_{n}+a_{0})\leq u^{+}(x)\leq U_{0}(x_{n}+b_{0})\quad \text{in $%
B_{r}(x_{0})\subset B_{2},$}  \label{oscdg}
\end{equation}%
with 
\begin{equation}  \label{G2}
\Vert u^{-}\Vert _{L^{\infty }}\leq \varepsilon ^{2},\quad \Vert f\Vert
_{L^{\infty }}\leq \varepsilon ^{4},\quad \Vert A-I\Vert _{C^{0,\bar \gamma
}}\leq \varepsilon ^{2}
\end{equation}%
and 
\begin{equation*}
b_{0}-a_{0}\leq \varepsilon r,
\end{equation*}%
for some $\varepsilon \leq \bar{\varepsilon},$ then 
\begin{equation*}
U_{0}(x_{n}+a_{1})\leq u^{+}(x)\leq U_{0}(x_{n}+b_{1})\quad \text{in $%
B_{r/20}(x_{0})$},
\end{equation*}%
with 
\begin{equation*}
a_{0}\leq a_{1}\leq b_{1}\leq b_{0},\quad b_{1}-a_{1}\leq (1-c)\varepsilon r,
\end{equation*}%
and $0<c<1$ universal.
\end{thm}

We can argue as in the nondegenerate case and get the following result.

\begin{cor}
\label{corollary4}Let $u$ be as in Theorem $\ref{HIdg}$ satisfying %
\eqref{oscdg} for $r=1$. Then in $B_1(x_0)$ 
\begin{equation*}
\tilde u_\varepsilon:= \frac{u^+(x) - x_n}{\varepsilon}
\end{equation*}
has a H\"older modulus of continuity at $x_0$, outside the ball of radius $%
\varepsilon/\bar \varepsilon,$ i.e for all $x \in B_1(x_0)$, with $|x-x_0|
\geq \varepsilon/\bar\varepsilon$ 
\begin{equation*}
|\tilde u_\varepsilon(x) - \tilde u_\varepsilon (x_0)| \leq C |x-x_0|^\gamma.
\end{equation*}
\end{cor}

As before, the proof of the Harnack inequality can be deduced from the
following lemma.

\begin{lem}
\label{main2}There exists a universal constant $\bar{\varepsilon}>0$ such
that if \eqref{G2} holds and $u$ satisfies 
\begin{equation*}
u^{+}(x)\geq U_{0}(x),\quad \text{in $B_{1}$}
\end{equation*}%
%
%
%
%
%
%
%
%
%with 
%\begin{equation}
%\Vert u^{-}\Vert _{L^{\infty }}\leq \varepsilon ^{2},\quad \Vert f\Vert
%_{L^{\infty }}\leq \varepsilon ^{4},\quad \Vert A-I\Vert _{C^{0,\gamma
%}}\leq \varepsilon ^{2}  \label{dg}
%\end{equation}%
for some $\varepsilon \leq \bar{\varepsilon},$and at $\bar{x}=\dfrac{1}{5}%
e_{n}$ 
\begin{equation}
u^{+}(\bar{x})\geq U_{0}(\bar{x}_{n}+\varepsilon ),  \label{u-p>ep2d}
\end{equation}%
then 
\begin{equation}
u^{+}(x)\geq U_{0}(x_{n}+c\varepsilon ),\quad \text{in $\overline{B}_{1/2},$}
\end{equation}%
for some $0<c<1$ universal. Analogously, if 
\begin{equation*}
u^{+}(x)\leq U_{0}(x),\quad \text{in $B_{1}$}
\end{equation*}%
and 
\begin{equation*}
u^{+}(\bar{x})\leq U_{0}(\bar{x}_{n}-\varepsilon ),
\end{equation*}%
then 
\begin{equation*}
u^{+}(x)\leq U_{0}(x_{n}-c\varepsilon ),\quad \text{in $\overline{B}_{1/2}.$}
\end{equation*}
\end{lem}

\begin{proof}
We prove the first statement. We use the same notation as in Lemma \ref{main}%
.

Since $x_n >0$ in $B_{1/10}(\bar x)$ and $u^+ \geq U_0$ in $B_1$ we get 
\begin{equation*}
B_{1/10}(\bar x) \subset B_1^+(u).
\end{equation*}

Thus $u-x_{n}\geq 0$ and solves $\mathcal{L}(u-x_{n})=f-\text{div}%
((A-I)e_{n})$ in $B_{1/10}(\bar{x})$ and we can apply Harnack inequality and
the assumptions \eqref{G2} and \eqref{u-p>ep2d} to obtain that (for $%
\varepsilon $ small enough) 
\begin{equation}
u-x_{n}\geq c_{0}\varepsilon \quad \text{in $\overline{B}_{1/20}(\bar{x})$}.
\label{u-p>cep2}
\end{equation}%
Let $w$ be as in the proof of Lemma \ref{main} and $\bar w =1-w$. Set, for $%
t\geq 0,$ 
\begin{equation*}
v_{t}(x)=(x_{n}-\varepsilon c_{0}\bar w +t\varepsilon )^{+}-\varepsilon
^{2}C_{1}(x_{n}-\varepsilon c_{0}\bar w +t\varepsilon )^{-},\quad x\in 
\overline{B}_{3/4}(\bar{x}).
\end{equation*}%
Here $C_{1}$ is a universal constant to be made precise later. We claim that 
\begin{equation*}
v_{0}(x)\leq u(x)\quad x\in \overline{B}_{3/4}(\bar{x}).
\end{equation*}

This is readily verified in the set where $u$ is non-negative using that $%
u^{+}\geq x_{n}^{+}.$ To prove our claim in the set where $u$ is negative we
wish to use the following fact: 
\begin{equation}
u^{-}\leq Cx_{n}^{-}\varepsilon ^{2},\quad \text{in $B_{\frac{19}{20}}$, $C$
universal}.  \label{negu}
\end{equation}%
This estimate is easily obtained using that $\{u<0\}\subset \{x_{n}<0\},$ $%
\Vert u^{-}\Vert _{\infty }<\varepsilon ^{2}$ and the comparison principle
with the function $z$ satisfying 
\begin{equation*}
\mathcal{L}z=-\varepsilon ^{4}\quad \text{in $B_{1}\cap \{x_{n}<0\}$},\quad
z=u^{-}\quad \text{on $\partial (B_{1}\cap \{x_{n}<0\})$.}
\end{equation*}%
Thus our claim immediately follows from the Lipschitz continuity of $z$ in $%
B_{19/20}\cap \{x_{n}<0\}$, the fact that $u^{-}\leq z$ and that, for $%
x_{n}<0$ and a suitable $C_{1}>C,$ 
\begin{equation*}
\varepsilon ^{2}C_{1}(x_{n}-\varepsilon c_{0}\bar w (x))\leq
Cx_{n}\varepsilon ^{2}.
\end{equation*}

Let $\bar{t}$ be the largest $t\geq 0$ such that, 
\begin{equation}
\{v_{t}>0\}\subset \{u>0\}\quad \text{in $\overline{D}_{7/8}.$}
\label{inclusion3}
\end{equation}%
%
%
%
%
%
%
%
%
%
%
% Thus at some
%$\overline B_{3/4}(\bar x)$ we have $$v_{\bar t}(\tilde x) =
%u(\tilde x).$$

From now on, the proof follows the lines of the non-degenerate case. Let $%
\phi _{t}^\pm$ and $D^\pm_{r,t}$ be as in \eqref{phit}, with $t \leq \bar t.$

Then, by assumption \eqref{G2} and the boundary regularity estimates for
divergence form equations we get 
\begin{equation}
\Vert \phi _{t}^+\Vert _{C^{1,\bar \gamma }} \leq C\varepsilon ^{2},\quad 
\text{in $\overline{D^{+}}_{6/7,t}$} \quad \Vert \phi _{t}^-\Vert
_{C^{1,\bar \gamma }} \leq C\varepsilon ^{4}\quad \text{in $\overline{D^{-}}%
_{6/7,t}$}  \label{esphi2}
\end{equation}%
and by $L^{\infty }$ estimates 
\begin{equation}
\Vert \phi^+ _{t}\Vert _{L^{\infty }} \leq C\varepsilon ^{2},\quad \text{in $%
\overline{D_{7/8,t}^{+}},$} \quad \Vert \phi^- _{t}\Vert _{L^{\infty }} \leq
C\varepsilon ^{4},\quad \text{in $\overline{D_{7/8,t}^{-}}.$}  \label{linf2}
\end{equation}
Call 
\begin{equation*}
\phi_{t} = 
\begin{cases}
\phi_{t}^+ \text{in $\overline{D^+_{7/8,t}}$} \\ 
\phi^-_{t} \text{in $\overline{D^-_{7/8,t}}.$} \\ 
\end{cases}%
\end{equation*}
Set 
\begin{equation*}
\psi _{t}=v_{t}+\phi _{t}.
\end{equation*}%
Since $w_{n}$ is bounded in the annulus $D_{7/8}$, we easily obtain that for 
$\varepsilon $ small enough, 
\begin{eqnarray*}
(v_{t})_{n} &\geq &c_{2}>0\text{ }\ \ \ \ \ \ \text{in }\overline{%
D_{7/8,t}^{+}} \\
(v_{t})_{n} &\geq &c_{3}\varepsilon ^{3}>0\quad \text{in $\overline{%
D_{7/8,t}^{-}}.$}
\end{eqnarray*}
Thus, from the $C^{1,\bar\gamma}$ estimates above we conclude that 
\begin{eqnarray*}
(\psi _{t})_{n} &\geq &c_{4}>0\text{ }\ \ \ \ \ \ \text{in }\overline{D^{+}}%
_{6/7,t} \\
(\psi _{t})_{n} &\geq &c_{5}\varepsilon ^{3}>0\quad \text{in }\overline{D^{-}%
}_{6/7,t}\text{$.$}
\end{eqnarray*}%
Hence, since $F(\psi _{t})$ is a graph in the $e_{n}$ direction for $%
\varepsilon $ small, we get that 
\begin{equation*}
\{\psi _{t}>0\}\cap \overline{D}_{6/7}=\{v_{t}>0\}\cap \overline{D}_{6/7},
\end{equation*}%
\begin{equation}
\{\psi _{t}<0\}\cap \overline{D}_{6/7}=\{v_{t}<0\}\cap \overline{D}_{6/7},
\label{fb1}
\end{equation}%
\begin{equation}
F(\psi _{t})\cap \overline{D}_{6/7}=F(v_{t})\cap \overline{D}_{6/7}.
\label{fb2}
\end{equation}%
Moreover, $\psi _{t}$ solves 
\begin{equation}
\mathcal{L}\psi _{t}=\Delta v_{t}\quad \text{in $D_{7/8,t}^{+}\cup
D_{7/8,t}^{-}$},  \label{int}
\end{equation}
and by assumption \eqref{G2} 
\begin{equation*}
\Delta v_{t}\geq \varepsilon ^{3}c_{0}k(n)>\varepsilon ^{4}\geq \Vert f\Vert
_{\infty },\quad \text{in $D_{7/8,t}^{+}\cup D_{7/8,t}^{-}$}
\end{equation*}%
for $\varepsilon $ small enough.

Let $t \leq \min\{\bar t, c_0\}$. Since $u^{+}\geq x_{n}^{+}$
and $w \leq -c <0$ in $\overline{D_{7/8,t}^{+}}\backslash D_{6/7}$, we have%
\begin{equation*}
v_{t}\left( x\right) =x_{n}-c_{0}\varepsilon \bar w +t\varepsilon \leq
x_{n}+c_{0}\varepsilon w \leq u - c\varepsilon.
\end{equation*}%
Hence, by the $L^\infty$ estimate \eqref{linf2} we have $\psi _{t}<u$
in $\overline{D_{7/8,t}^{+}}\backslash D_{6/7}$ for $\varepsilon $ small$.$

Moreover, since $\mathcal{L}\psi _{{t}}\geq f,$ the maximum principle
gives $\psi _{t} \leq u$ in $\overline{D}_{7/8,t}^{+}$ (using also \eqref{u-p>cep2}).

In $D_{7/8,t}^{-}\backslash \overline{D}_{6/7}\cap \left\{ u<0\right\}
,$ using (\ref{negu}), since $\{u<0\}\subset \{x_{n}<0\}$, \begin{equation*}
v_{{t}}\left( x\right) -u=-\varepsilon ^{2}C_{1}(x_{n}-c_{0}\varepsilon
\bar w +{t}\varepsilon )^{-}+u^{-}\leq -c\varepsilon ^{3}
\end{equation*}%
if $C_{1}$ is chosen large enough. Hence, by the $L^{\infty }$ estimates for 
$\phi _{t},$ we have $\psi _{{t}}<u$ in $\overline{D_{7/8,t}^{-}}\backslash 
D_{6/7}\cap \left\{ u<0\right\} $ for $\varepsilon $ small$.$
Again by maximum principle we infer $\psi _{t}\leq u,$ in $\{u<0\}\cap
D_{6/7,t}^{-}.$ 

In particular,%
\begin{equation*}
\mathcal{L}\psi _{t}>f,\quad \psi _{\bar{t}}\leq u,\quad \text{in }\{\psi
_{t}<0\}\cap \overline{D}_{6/7}
\end{equation*}%
Summarizing: if $t \leq \min\{\bar t, c_{0}\}$ then $\psi _{t}\leq u,$ in $%
\overline{D}_{6/7}$ and $v _{\bar{t}}<u$ in $\overline{D}_{7/8}\backslash 
D_{6/7}$ and on $\partial B_{1/20}\left( \bar{x}\right) $ by (\ref%
{u-p>cep2}).

On the other hand, using the $C^{1,\bar \gamma}$ estimates for $\phi_t^\pm$
and that on $F(v_{t})\cap D_{7/8}$ 
\begin{equation*}
|\nabla v_{t}^{+}|^2=(1+\varepsilon ^{2}c_{0}^{2}|\nabla w
|^{2}+2\varepsilon c_{0}w_n ), \quad |\nabla v_{t}^{-}|^2=\varepsilon
^{4}C_{1}^{2}(1+\varepsilon ^{2}c_{0}^{2}|\nabla w |^{2}-2\varepsilon
c_{0}w_n),
\end{equation*}
we conclude as in Lemma \ref{main} that 
\begin{equation*}
|\nabla_A\psi _{t}^{+}|^{2}- |\nabla_A \psi _{t}^{-}|^{2}>1\quad \text{on $%
F(\psi_{t})\cap $}\overline{D}_{6/7}.
\end{equation*}
We thus reach a contradiction as in the non-degenerate case, unless $\bar{t}%
>c_{0}.$

In particular, 
\begin{equation}
\psi _{c_0}\leq u\quad \text{in }D_{6/7}  \label{c_0}
\end{equation}
and we can write, by the $L^{\infty }$ estimate for $\phi _{c_0},$ in $%
B_{1/2}\subset \subset $$\overline{D}_{6/7}$ 
\begin{eqnarray*}
u^{+}(x) &\geq &\psi _{c_0}^{+}(x)\geq -C\varepsilon ^{2}+\left(
x_{n}-\varepsilon c_{0}\psi + c_0\varepsilon \right) ^{+} \\
&=&-C\varepsilon ^{2}+\left( x_{n}+\varepsilon c_0 w \right) ^{+} \\
&\geq &-C\varepsilon ^{2}+\left( x_{n}+\varepsilon \bar c \right) ^{+} \\
\end{eqnarray*}%
where we used that $w \geq c >0$ on $B_1/2.$ Since $u^{+}\left( x\right)
\geq x_{n}^{+}$ we infer, if $\varepsilon $ is small, 
\begin{equation*}
u^{+}(x)\geq \left( x_{n}+c\varepsilon \right) ^{+},
\end{equation*}%
with $c$ universal.
\end{proof}

\subsection{Improvement of flatness}

We are now ready to prove the improvement of flatness lemma in the
degenerate case.

\begin{lem}
\label{improv4_deg}Let $u$ satisfy
\begin{equation}
U_{0}(x_{n}-\varepsilon )\leq u^{+}(x)\leq U_{0}(x_{n}+\varepsilon )\quad 
\text{in $B_{1},$}\quad 0\in F(u),  \label{flat***}
\end{equation}%
with 
\begin{equation*}
\Vert f\Vert _{L^{\infty }(B_{1})}\leq \varepsilon ^{4},\Vert A-I\Vert
_{C^{0,\bar \gamma }}\leq \varepsilon ^{2}
\end{equation*}%
and 
\begin{equation*}
\Vert u^{-}\Vert _{L^{\infty }(B_{1})}\leq \varepsilon ^{2}.
\end{equation*}

If $0<r\leq r_{1}$ for $r_{1}$ universal, and $0<\varepsilon \leq
\varepsilon _{1}$ for some $\varepsilon _{1}$ depending on $r$, then 
%\begin{equation}\label{improvedflat_1}(x \cdot \nu -r\frac{\ep^2}{2} -C r^2\ep)^+ \leq u(x) \leq
%(x \cdot \nu +r\frac{\ep^2}{2} +C r^2\ep)^+ \quad \text{for $x \in
%B_r$}
%\end{equation}

%In particular, if $C r \leq 1/4$ and $\ep_0 \leq 1/2$ then
\begin{equation}
U_{0}(x\cdot \nu _{1}-r\frac{\varepsilon }{2})\leq u^{+}(x)\leq U_{0}(x\cdot
\nu _{1}+r\frac{\varepsilon }{2})\quad \text{in $B_{r},$}
\label{improvedflat_2}
\end{equation}%
with $|\nu _{1}|=1,$ $|\nu _{1}-e_{n}|\leq C\varepsilon $ for a universal
constant $C.$
\end{lem}

\begin{proof}
We argue similarly as in the non-degenerate case.

\vspace{2mm}

\textbf{Step 1 -- Compactness.} Fix $r\leq r_{1}$ with $r_{1}$ universal
(made precise in Step 3). Assume by contradiction that we can find a
sequence $\varepsilon _{k}\rightarrow 0$ and a sequence $u_{k}$ of solutions
to 
\begin{equation*}
\begin{cases}
\mathcal{L}_{k}u_{k}:=\text{div}(A_{k}\nabla u_{k})=f_{k}\quad \text{in $%
B_{1}^{+}(u_{k})$}\cup B_{1}^{-}(u_{k}) \\ 
(u_{k}^{+})_{\nu }^{2}-(u_{k}^{-})_{\nu }^{2}=1\quad \text{on $F(u_{k})$}%
\end{cases}%
\end{equation*}%
\begin{equation*}
\Vert u_{k}^{-}\Vert _{L^{\infty }(B_{1})}\leq \varepsilon_k^{2}
\end{equation*}%
with 
\begin{equation*}
\Vert f_{k}\Vert _{L^{\infty }}\leq \varepsilon _{k}^{4},\quad \Vert
A_{k}-I\Vert _{C^{1,\bar \gamma }}\leq \varepsilon _{k}^{2}
\end{equation*}%
such that 
\begin{equation}
U_{0}(x_{n}-\varepsilon _{k})\leq u_{k}^{+}(x)\leq U_{0}(x_{n}+\varepsilon
_{k})\quad \text{for $x\in B_{1}$, $0\in F(u_{k}),$}  \label{kstrip}
\end{equation}%
but $u_{k}$ does not satisfy the conclusion of the lemma.

Set 
\begin{equation*}
\tilde{u}_{k}(x)=\dfrac{u_{k}(x)-x_{n}}{\varepsilon _{k}},\quad x\in
B_{1}^{+}(u_{k})\cup F(u_{k})
\end{equation*}%
Then (\ref{kstrip}) gives, 
\begin{equation}
-1\leq \tilde{u}_{k}(x)\leq 1\quad \text{for $x\in B_{1}^{+}(u_{k})\cup
F(u_{k})$}.  \label{flat_tilde*}
\end{equation}

As in the non-degenerate case, it follows from Corollary \ref{corollary4}
that as $\varepsilon _{k}\rightarrow 0$ the graphs of the $\tilde{u}_{k}$
converge (up to a subsequence) in the Hausdorff distance to the graph of a H%
\"{o}lder continuous function $\tilde{u}$ over $B_{1/2}\cap \{x_{n}\geq 0\}$.

\medskip

\textbf{Step 2 -- Limiting Solution.} We now show that $\tilde{u}$ solves
the following Neumann problem 
\begin{equation}
\begin{cases}
\Delta \tilde{u}=0 & \text{in $B_{1/2}\cap \{x_{n}>0\}$}, \\ 
\  &  \\ 
\partial _{n}\tilde{u}=0 & \text{on $B_{1/2}\cap \{x_{n}=0\}$}.%
\end{cases}
\label{Neumann4}
\end{equation}

As before, the interior condition follows easily. Thus we focus on the
boundary condition.

Let $\tilde{\phi}$ be a function of the form 
\begin{equation*}
\tilde{\phi}(x)=M+px_{n}+BQ(x-y)
\end{equation*}%
with 
\begin{equation*}
Q(x)=\frac{1}{2}[(n-1)x_{n}^{2}-|x^{\prime }|^{2}],\quad y=(y^{\prime
},0),\quad M\in \R,B>0
\end{equation*}%
and 
\begin{equation*}
p>0.
\end{equation*}%
Then we must show that $\tilde{\phi}$ cannot touch $u$ strictly by below at
a point $x_{0}=(x_{0}^{\prime },0)\in B_{1/2}$. Suppose that such a $\tilde{%
\phi}$ exists and let $x_{0}$ be the touching point.

Let $\Gamma_k$ be as in the proof of the non-degenerate case (see %
\eqref{biggammak}). Call

\begin{equation*}
\phi_k(x)= a_k \Gamma^+_k(x)+ (d_k^+(x))^2\varepsilon_k^2, \quad
a_k=(1+\varepsilon_k p)
\end{equation*}
where $d_k(x)$ is the signed distance from $x$ to $\partial B_{\frac{1}{%
B\varepsilon_k}}(y+e_n(\frac{1}{B\varepsilon_k}-M\varepsilon_k)).$

Now, call $g_{k}$ the solution to

\begin{equation*}
\begin{cases}
\mathcal{L}_{k}g_{k}=-\text{div}((A_{k}-I)\nabla \phi _{k})\quad \text{in $%
B_{1}^{+}(\phi _{k})$} \\ 
g_{k}=0\quad \text{on $\partial B_{1}^{+}(\phi _{k})$}%
\end{cases}%
\end{equation*}%
extended by zero to $B_{1}^{-}(\phi _{k})$

By $C^{1,\bar\gamma }$ and $L^{\infty }$ estimates, using the first one of %
\eqref{ass1-ass2} and the formula for $\phi _{k}$, we get

\begin{equation*}
\Vert g_{k}\Vert _{C^{1,\bar\gamma }}\leq C(a_{k}+1)\varepsilon
_{k}^{2}\quad \text{in $\overline{B}_{7/8}^{+}(\phi _{k})$,}
\end{equation*}%
\begin{equation*}
\|g_k\|_{L^\infty} \leq C(a_{k}+1)\varepsilon _{k}^{2}\quad \text{in $%
\overline{B}_{1}^{+}(\phi _{k})$.}
\end{equation*}
Set 
\begin{equation*}
\tilde{\phi}_{k}(x)=\dfrac{\phi _{k}(x)-x_{n}}{\varepsilon _{k}},
\end{equation*}
\begin{equation*}
\gamma_k := \phi_k + g_k,
\end{equation*}
and%
\begin{equation*}
\tilde{\gamma}_{k}(x)=\tilde{\phi}_{k}+\dfrac{g_{k}}{\varepsilon _{k}}=%
\dfrac{\phi _{k}(x)-x_{n}}{\varepsilon _{k}}+\dfrac{g_{k}}{\varepsilon _{k}}.
\end{equation*}

As shown in \cite{DFS1}, (the graph of) $\tilde{\phi}_{k}$ converges
uniformly to (the graph of) $\tilde{\phi}$ on $B_{1/2}\cap \left\{
x_{n}>0\right\} $.

Thus, from the $L^{\infty }$ estimates for $g_{k}$, also $\tilde{\gamma}_{k}$
converges uniformly to $\tilde{\phi}$ on $B_{1/2}.$ Since $\tilde{u}_{k}$
converges uniformly to $\tilde{u}$ and $\tilde{\phi}$ touches $\tilde{u}$
strictly by below at $x_{0} $, we conclude that there exist a sequence of
constants $c_{k}\rightarrow 0$ and of points $x_{k}\rightarrow x_{0}$ such
that the function 
\begin{equation*}
\psi _{k}(x)=\gamma _{k}(x+\varepsilon _{k}c_{k}e_{n})
\end{equation*}%
touches $u_{k}$ by below at $x_{k}\in B_{1}^{+}(u_{k})\cup F(u_{k})$. We
claim that $x_{k}$ cannot belong to $B_{1}^{+}(u_{k})$. Otherwise, in a
small neighborhood $N$ of $x_{k}$ we would have that 
\begin{equation*}
\mathcal{L}_{k}\psi _{k}=\Delta \phi _{k}>\varepsilon _{k}^{4}\geq f_{k}=%
\mathcal{L}_{k}u_{k},\quad \text{$\psi _{k}<u_{k}$ in $N\setminus
\{x_{k}\},\psi _{k}(x_{k})=u_{k}(x_{k}),$}
\end{equation*}%
a contradiction.

Thus $x_{k}\in F(u_{k})\cap \partial B_{\frac{1}{B\varepsilon _{k}}}(y+e_{n}(%
\frac{1}{B\varepsilon _{k}}-M\varepsilon _{k}-\varepsilon _{k}c_{k})).$ For
simplicity we call 
\begin{equation*}
\mathcal{B}_{k}:=B_{\frac{1}{B\varepsilon _{k}}}(y+e_{n}(\frac{1}{%
B\varepsilon _{k}}-M\varepsilon _{k}-\varepsilon _{k}c_{k})).
\end{equation*}%
Let $N_{\rho }$ be a small neighborhood of $x_{k}$ of size $\rho $. Since 
\begin{equation*}
\Vert u_{k}^{-}\Vert _{\infty }\leq \varepsilon _{k}^{2},\quad u_{k}^{+}\geq
(x_{n}-\varepsilon _{k})^{+},
\end{equation*}%
as in the proof of Harnack inequality, using the fact that $x_{k}\in
F(u_{k})\cap \partial \mathcal{B}_{k}$ we can conclude by the comparison
principle that

\begin{equation*}
u_{k}^{-}\left( x\right) \leq c\varepsilon _{k}^{2}W_{k},\quad \text{in }%
N_{\rho }\backslash \mathcal{B}_{k}
\end{equation*}%
where $\mathcal{L}_kW_{k}=-1$ in $N_{\rho }\backslash \mathcal{B}_{k}\,,$ $%
W_k=0$ on $N_{\rho }\cap \partial \mathcal{B}_{k},W_k=1$ on $\partial
N_{\rho }\backslash \mathcal{B}_{k}.$

Let 
\begin{equation}
\Psi _{k}(x)=\left\{ 
\begin{array}{ll}
\psi _{k}(x) & \text{ \ \ in }\mathcal{B}_{k}\cap N_{\rho } \\ 
-\varepsilon _{k}^{2}W_k & \text{ \ \ in }N_{\rho }\backslash \mathcal{B}_{k}%
\end{array}%
\right.  \label{Psi}
\end{equation}

Then $\mathcal{L}_{k}\Psi _{k}\geq f_{k}$ in $\mathcal{B}_{k}\cap N_{\rho }$
and $N_{\rho }\backslash \mathcal{B}_{k}.$

We reach a contradiction if we show that 
\begin{equation*}
\mid \nabla_{A_k}\Psi _{k}^{+}\mid^{2}-\mid \nabla_{A_k}\Psi _{k}^{-}\mid
^{2}>1,\quad \hbox{on
$F(\Psi_k)$.}
\end{equation*}
Since 
\begin{equation*}
\mid \nabla_{A_k}\Psi _{k}^{+}\mid^{2}-\mid \nabla_{A_k}\Psi _{k}^{-}\mid ^{2} \geq
(1-\varepsilon_k^2)|\nabla \Psi_k^+|^2-(1+\varepsilon_k^2)|\nabla
\Psi_k^-|^2,
\end{equation*}
this follows from the formula for $\Psi_k$ for $k$ large enough, because $p>0.$ We finally reached a
contradiction.

\medskip

\textbf{Step 3 -- Contradiction.} In this step we can argue as in the final
step of the proof of Lemma 4.1 in \cite{D}.
\end{proof}

\section{The proof of Theorem \protect\ref{Lipmainvar}}

In this section we provide the proof of Theorem \ref{Lipmainvar}. The proof
follows via a blow-up argument and our flatness Theorem \ref{flatmain2}, as
in Section 6.2 of \cite{DFS1}. The extra ingredient in that argument was the
regularity theory developed by Caffarelli in \cite{C1} in the homogeneous
case. Here we provide a different proof of that result, based on a Weiss
type monotonicity formula and our flatness Theorem \ref{flatmain2}. The same
strategy has been employed in \cite{DS}.

Precisely, we have the following result.

\begin{thm}
\label{new_reg} \label{Lipmainvar2} Let $u$ be a viscosity solution to 
\begin{equation}  \label{fbvisco}
\begin{array}{ll}
\Delta u = 0, & \hbox{in $B_1(u^+) \cup B_1(u^-),$} \\ 
\  &  \\ 
(u_\nu^+)^2 - (u_\nu^-)^2= 1, & \hbox{on $F(u)$} \\ 
& 
\end{array}%
\end{equation}
with $0\in F(u)$. If $F(u)$ is a Lipschitz graph in $B_1$, then $F(u)$ is $%
C^{1,\gamma}$ in $B_{1/2}$, with norm controlled by a universal constant.
\end{thm}

Let 
\begin{equation*}
E(u, r) : = \int_{B_r} (|\nabla u|^2 + \alpha^2\chi_{\{u>0\}} +
\beta^2\chi_{\{u<0\}}) dx, \quad \alpha^2-\beta^2=1,
\end{equation*}
and define 
\begin{equation}  \label{mon_phi}
\Phi_u(r) := r^{-n} E(u,r) - r^{-1-n}\int_{\partial B_r} u^2 d\mathcal H^{n-1}.
\end{equation}

The proof of Theorem \ref{new_reg} is based on a Weiss-Type monotonicity
formula for the function $\Phi_u$. In the case when $u$ is a critical point
for the energy functional $E$ (with respect to domain variations), then the
proof of this formula is contained in \cite{W}. In our context, we need a
formula for viscosity solutions. Before the proof we remark that the
rescaling 
\begin{equation*}
u_\lambda(X) := \lambda^{-1} u (\lambda X)
\end{equation*}
satisfies 
\begin{equation}  \label{rescaledmon}
\Phi_{u_\lambda}(r) = \Phi_u(\lambda r).
\end{equation}

\begin{thm}
Let $u$ be a viscosity solution to \eqref{fbvisco} in $B_1$ and assume that $%
F(u)$ is a Lipschitz graph. Then $\Phi_u(r)$ is monotone increasing for $0<r
\leq 1.$ Moreover $\Phi_u$ is constant if and only if $u$ is homogeneous of
degree 1.
\end{thm}

\begin{proof}
%Since $u$ is a Lipschiz function in $\Omega,$ $\Delta u=0$ in $\Omega^+(u),$ $\Delta u=0$ in $\Omega^-(u)$ and moreover $u=0$ on $F(u)$  in a viscosity sense, where $F(u)$ is a Lipschitz set. Then $u\in H^{1,2}(\Omega).$ 

First observe, by our flatness Theorem \ref{flatmain2}, that the free
boundary condition is satisfied almost everywhere on $F(u).$ As mentioned
before, by Theorem 3.2 \cite{W} it is sufficient to prove that $u$ is a
critical point for the energy functional $E,$ with respect to domain
variation. Precisely, for every $\phi \in C_0^{1}(B_1,\mathbb{R}^n)$ $u$
satisfies 
\begin{equation}
\begin{split}
&0=-\frac{d}{d\epsilon}E(u(x+\epsilon\phi(x)))_{\vert\epsilon=0} \\
&=\int_{B_1}\left(\mid \nabla u\mid^2\mbox{div}\phi-2\nabla u D\phi\nabla
u+\alpha^2\chi_{\{u>0\}}\mbox{div}\phi+\beta^2\chi_{\{u<0\}}\mbox{div}%
\phi\right).
\end{split}%
\end{equation}

Let $u$ be our viscosity solution and $\phi \in C_0^{1}(B_1,\mathbb{R}^n)$.
Call 
\begin{equation*}
H(u):=\int_{B_1}\left(\mid \nabla u\mid^2\mbox{div}\phi-2\nabla u
D\phi\nabla u+\alpha^2\chi_{\{u>0\}}\mbox{div}\phi+\beta^2\chi_{\{u<0\}}%
\mbox{div}\phi\right).
\end{equation*}
An easy computation shows that in $B_1^\pm(u)$, 
\begin{equation*}
\langle \nabla |\nabla u|^2, \phi \rangle + 2 \nabla u D\phi \nabla u = 2
\langle \nabla\langle \phi, \nabla u \rangle, \nabla u \rangle.
\end{equation*}
Then, integrating by parts and using that $\phi$ is compactly supported in $%
B_1$ we get 
\begin{align*}  \label{weir}
&H(u)=-\int_{B_1\cap\{u>0\}}\left(\langle\nabla\mid \nabla
u\mid^2,\phi\rangle+2\nabla u D\phi\nabla u\right)+\int_{F(u)}\mid\nabla
u^+\mid^2\langle\phi,\nu\rangle \\
&-\int_{B_1\cap\{u<0\}}\left(\langle\nabla\mid \nabla
u\mid^2,\phi\rangle+2\nabla u D\phi\nabla u\right)-\int_{F(u)}\mid\nabla
u^-\mid^2\langle\phi,\nu\rangle \\
&+\int_{F(u)}(\alpha^2- \beta^2)\langle\phi,\nu\rangle \\
&=-2\int_{B_1\cap\{u>0\}}\langle\nabla\langle\phi,\nabla u\rangle,\nabla
u\rangle+\int_{F(u)}(\mid\nabla u^+\mid^2-\mid\nabla
u^-\mid^2)\langle\phi,\nu\rangle \\
&-2\int_{B_1\cap\{u<0\}}\langle\nabla\langle\phi,\nabla u\rangle,\nabla
u\rangle +\int_{F(u)}\langle\phi,\nu\rangle,
\end{align*}
where $\nu$ denotes the unit normal vector to $F(u)$ pointing towards $%
B^+_1(u).$ In the last equality we used that $\alpha^2-\beta^2=1.$

Integrating by parts again, and using that $u$ is harmonic in $B^\pm_1(u)$
and $\phi$ is compactly supported in $B_1$ we get 
\begin{align*}
&H(u)=-2\int_{F(u)} \langle\phi,\nabla u^+\rangle\langle\nabla u^+,
\nu\rangle+\int_{F(u)}(\mid\nabla u^+\mid^2-\mid\nabla
u^-\mid^2)\langle\phi,\nu\rangle \\
&+2\int_{F(u)} \langle\phi,\nabla u^-\rangle\langle\nabla u^-, \nu\rangle
+\int_{F(u)}\langle\phi,\nu\rangle.
\end{align*}

Since $\nu = \frac{\nabla u^+}{\mid\nabla u^+\mid}= - \frac{\nabla u^-}{%
\mid\nabla u^-\mid}$ a.e. on $F(u)$ we get that 
\begin{align*}
H(u)=-\int_{F(u)}(\mid\nabla u^+\mid^2 - \mid\nabla
u^-\mid^2)\langle\phi,\nu\rangle +\int_{F(u)}\langle\phi,\nu\rangle =0,
\end{align*}
because the free boundary condition is satisfied a.e.
\end{proof}

\begin{rem}
\label{rem}If $u_k$ are viscosity solutions with Lipschitz free boundaries
with uniform Lipschitz bound which converges uniformly to $u$ on compact
sets, then it follows that 
\begin{equation}  \label{limit}
\Phi_{u_k} (r) \to \Phi_u(r).
\end{equation}

Moreover, if $u$ satisfies the assumptions of Theorem \ref{Lipmainvar} then $%
\Phi_u(r)$ is bounded below as $r \to 0$. This means that 
\begin{equation*}
\Phi_u(0^+) = \lim_{r \to 0^+} \Phi_u(r) \quad \text{exists}
\end{equation*}
and by \eqref{limit}-\eqref{rescaledmon} any blow-up sequence $u_\lambda$
converges uniformly on compact sets (up to a subsequence) to a homogeneous
of degree $1$ solution $u_0$.

\begin{rem}
\label{ACF}By the monotonicity formula of Alt-Caffarelli-Friedman \cite{ACF}%
, either $u_0$ is a two-plane solution $U_\beta$ for $\beta > 0$ or $u_0^-
\equiv 0$.
\end{rem}

Consider the one-phase problem: 
\begin{equation}  \label{fbvisco1}
\begin{array}{ll}
\Delta U = 0, & \hbox{in $\{U>0\}$} \\ 
\  &  \\ 
|\nabla U|^2= 1, & \hbox{on $F(U).$} \\ 
& 
\end{array}%
\end{equation}
\end{rem}

\begin{defn}
\label{def-cones} A global viscosity solution to \eqref{fbvisco1} which is
homogeneous of degree 1 and has Lipschitz free boundary is called a \textit{%
Lipschitz cone.}
\end{defn}

We say that a Lipschitz cone is trivial if it coincides (up to rotations)
with the one-plane solution $U_0=x_n^+$. We wish to prove the following
theorem.

\begin{thm}
\label{trivial} All Lipschitz cones are trivial.
\end{thm}

To this aim, we will use a standard dimension reduction argument. A point $%
x_0 \in F(U)$ such that $F(U)$ is $C^{1,\gamma}$ in a neighborhood of $x_0,$
is called regular. Points that are not regular, are called singular.

\begin{lem}
\label{dimred}Assume $U$ is a Lipschitz cone in $\R^{n}$ and $x_0=e_1 \in F(U)$%
. Then, any blow-up sequence 
\begin{equation*}
V_\lambda(x) = \lambda^{-1} U(x_0+\lambda x)
\end{equation*}
has a subsequence $V_{\lambda_k}, \lambda_k \to 0$ which converges uniformly
on compact sets to $V(x_2,\ldots,x_{n})$ with $V$ a Lipschitz cone in $%
\R^{n-1} $. Moreover if $x_0$ is a singular point for $F(U)$, then $V$ is a
non-trivial cone.
\end{lem}

\begin{proof}
From the fact that $U$ is homogeneous of degree 1 and from the formula for $%
V_\lambda$ we get that 
\begin{align*}
V_\lambda(x) &= \lambda^{-1} (1+t\lambda)^{-1} U((1+ t\lambda)(x_0+\lambda
x)) \\
&= (1+t\lambda)^{-1}V_\lambda(t x_0 +(1+t\lambda)x).
\end{align*}

Letting $\lambda=\lambda_k \to 0$ we obtain that 
\begin{equation*}
V(x) = V(t x_0 +x), \quad \text{for all $t$}.
\end{equation*}
Thus, $V$ is constant in the $x_0=e_1$ direction and by Remark \ref{rem} is
homogeneous of degree 1. Now, it is easily checked from the definition that $%
V(x_2,\ldots, x_n)$ is a viscosity solution in $\R^{n-1},$ and clearly $V$ is
a Lipschitz cone.

The final statement follows from the flatness Theorem \ref{flatmain2}.
\end{proof}

Assume that $U$ is a non-trivial Lipschitz cone in $\R^{n}$ for some dimension 
$n$. Then by Lemma \ref{dimred} we obtain that if $F(U)$ has a singular
point different than the origin, then there exists a non-trivial Lipschitz
cone in $\R^{n-1}.$ By repeating this dimension reduction argument, we can
assume that there is a dimension $k<n$ and a non-trivial cone in $\R^{k+1}$
which is regular at all points except at 0. Thus, Theorem \ref{trivial}
reduces to the following proposition.

%Clearly, all minimal cones in dimension $n=1$ are trivial. 

\begin{prop}
\label{Lip_trivial}All Lipschitz viscosity cones whose free boundary is $%
C^{\infty }$ outside of the origin are trivial.
\end{prop}

\begin{proof}
Let $U$ be a Lipschitz viscosity cone which is smooth outside the origin,
and denote by $L$ the Lipschitz norm of $F(U)$ as a graph in the $e_n$
direction. We want to show that $U$ is trivial.

We prove the proposition by induction on $n$. The case $n=1$ is obvious.
Assume the statement holds for $n-1.$

By Proposition \ref{mon_prop} below, $U$ is monotone in the cone of
directions 
\begin{equation*}
\mathcal C:= \{\xi=(\xi^{\prime}, \xi_n) \in \mathbb{R}^n: \:\:  \xi_n \geq L |\xi^{\prime}|\},
\end{equation*}
since $F(U)$ is a Lipschitz graph with respect to any direction $\xi \in \mathcal C^o.$
Moreover there is a direction $\tau \in \partial \mathcal{C}$, $|\tau| =1$ such that $
\tau $ is tangent to $F(U)$ at some point $X_0 \in F(U) \setminus \{0\}.$
Then, 
\begin{equation*}
U_\tau \geq 0 \quad \text{in $\{U>0\}$}.
\end{equation*}

If $U_\tau=0$ at some point in $\{U>0\}$ then $U_\tau \equiv 0$, thus $U$ is
constant in the $\tau$ direction, and by dimension reduction we can reduce
the problem to $n-1$ dimensions thus by the induction assumption $U$ is
trivial. Otherwise $U_\tau >0$ in $\{U>0\}$ and by Hopf Lemma 
\begin{equation*}
U_{\tau \nu} >0 \quad \text{on $F(U) \setminus \{0\}.$}
\end{equation*}
This contradicts the free boundary condition, $U_\nu^2=1$ on $F(U) \setminus
\{0\}$.
\end{proof}

In the proof above we used the following result. Its proof is contained for
example in \cite{CS} and it is a consequence of the Boundary Harnack
Inequality.

\begin{prop}
\label{mon_prop}Assume that $v\geq 0$ solves $\Delta v=0$ in $B^+_1(v)$, and
that $F(v)$ is a Lipschitz graph in the $e_n$ direction in $B_1$ with
Lipschitz constant $L,$ and $0 \in F(v).$ Then $v$ is monotone in the $e_n$
direction in $B_\delta$, with $\delta$ depending on $L$ and $n$.
\end{prop}

We are now finally ready to exhibit the proof of Theorem \ref{new_reg}.

\ 

\textit{Proof of Theorem $\ref{new_reg}$.} First, we show that given a
viscosity solution $u$ with Lipschitz free boundary in $B_1$, $0 \in F(u)$,
we can find $\sigma>0$ small depending on $u$ such that $F(u)$ is a $%
C^{1,\gamma}$ graph in $B_{\sigma}.$ Indeed, there exists a blow-up sequence 
$u_{\lambda_k}$ which converges to a Lipschitz viscosity cone (see Remark %
\ref{rem}), that in view of Remark \ref{ACF} and Theorem \ref{trivial} is of
the form $U_\beta$ for $\beta \geq 0$. The conclusion now follows from our
flatness Theorem \ref{flatmain2}.

Next we use compactness to show that $\sigma$ depends only on the Lipschitz
constant $L$ of $F(u).$ For this we need to show that $F(u)$ is $\bar
\varepsilon$-flat in $B_r$ for some $r \geq \sigma$ depending on $L$. If by
contradiction no such $\sigma$ exists, then we can find a sequence of
solutions $u_k$ and of $\sigma_k \to 0$ such that $u_k$ is not $\bar
\varepsilon$-flat in any $B_r$ with $r \geq \sigma_k.$ Then the $u_k$
converge uniformly (up to a subsequence) to a solution $u_*$ and we reach a
contradiction since $F(u_*)$ is $C^{1,\gamma}$ in a neighborhood of 0 by the
first part of the proof. \qed

\section{Perron's solutions}

In this section we apply our results to the Perron's solution constructed in 
\cite{DFS5}, where we used the following definition of weak or viscosity
solution. Given a continuous function $v$ on $\Omega $, we say that a point $%
x_{0}\in F(v)$ is regular from the right (resp. left) if there is a ball $%
B\subset \Omega ^{+}(v)$ (resp. $B\subset \Omega ^{-}(v)$), such that $%
\overline{B}\cap F(v)=\{x_{0}\}$. Let us denote by $\nu (x_{0}) $ the unit
normal at the point $x_{0}.$ $A\left( x_{0}\right) \nu (x_{0})$ is the
co-normal to $\partial B$ at $x_{0}$ pointing toward $\Omega ^{+}(v).$ For
coherence with the rest of the paper, in what follows 
\begin{equation*}
G(\beta ,\nu ,x)=\sqrt{\langle A\nu ,\nu \rangle ^{-1}+\beta ^{2}}.
\end{equation*}%
However the reader should notice that the arguments below continue to work
for a general $G$, as considered in \cite{DFS5}.

\begin{defn}
\label{Per} A function $u\in C(\Omega)$ is a weak solution of \eqref{FBintro}
if

\begin{enumerate}
\item $\mathcal{L}u=f$ in $\Omega^+(u)\cup\Omega^-(u)$ in the weak sense;

\item 
\begin{itemize}
\item[(a)] If $x_0\in F(u)$ is regular from the right with touching ball $B,$
then in a neighborhood of $x_0$ 
\begin{equation*}
u^+\geq \alpha\langle x-x_0,\nu(x_0)\rangle^++o(|x-x_0|),\quad\alpha\geq 0,
\end{equation*}
in $B$ and 
\begin{equation*}
u^-\leq \beta\langle x-x_0,\nu(x_0)\rangle^-+o(|x-x_0|),\quad\beta\geq 0,
\end{equation*}
in $B^c$ with equality along non-tangential domain and 
\begin{equation*}
\alpha\leq G(\beta,\nu_0, x_0).
\end{equation*}

\item[(b)] If $x_0\in F(u)$ is regular from the left with touching ball $B,$
then in a neighborhood of $x_0$ 
\begin{equation*}
u^-\geq \beta\langle x-x_0,\nu(x_0)\rangle^++o(|x-x_0|),\quad\mbox{in}\:\:\:
B,\:\:\:\mbox{with}\:\:\: \beta\geq 0,
\end{equation*}
and 
\begin{equation*}
u^+\leq \alpha\langle x-x_0,\nu(x_0)\rangle^++o(|x-x_0|),\quad\mbox{in}%
\:\:\: B^c,\:\:\:\mbox{with}\:\:\: \alpha\geq 0,
\end{equation*}
with equality along non-tangential domain and 
\begin{equation*}
\alpha\geq G(\beta,\nu_0,x_0).
\end{equation*}
\end{itemize}
\end{enumerate}
\end{defn}

\begin{lem}
Definition $\ref{Per}$ and Definition $\ref{defnhsol}$ are equivalent.
\end{lem}

\begin{proof}
Let $u$ be a weak solution as in Definition \ref{Per}. Assume $%
v\in C^{1,{\bar{\gamma}}}(\overline{B^{+}(v)})\cap C^{1,\bar{\gamma}}(\overline{B^{-}(v)}%
),$ with $F(v)\in C^{2},$ touches $u$ from below at $x_{0}\in F(v).$ Then we
need to show 
\begin{equation*}
|\nabla v^{+}(x_{0})|\leq G(|\nabla v^{-}(x_{0})|,\nu _{0},x_{0}).
\end{equation*}%
Indeed, $v\leq u$ in a neighborhood of $x_{0}$ and $x_{0}$ is regular from
the right. Thus 
\begin{equation*}
u^{+}\geq |\nabla v^{+}(x_{0})|\langle \nu (x_{0}),x-x_{0}\rangle
^{+}+o(|x-x_{0}|),
\end{equation*}%
\begin{equation*}
u^{-}\leq |\nabla v^{-}(x_{0})|\langle \nu (x_{0}),x-x_{0}\rangle
^{-}+o(|x-x_{0}|),
\end{equation*}%
and by part (ii)-(b) in Definition \ref{Per}, we obtain the desired
inequality. A similar argument works for test functions touching $u$ from
above.

Assume now that $u$ is a viscosity solution in the sense of Definition \ref%
{defnhsol} (we consider for simplicity only the case (ii)-(a)). 
If $x_{0}\in F(u)$ is regular from the right with touching ball $B,$ then,
by Lemmas 2.2 and 2.3 in \cite{DFS5}, we can write 
\begin{equation*}
u^{+}\geq \alpha \langle x-x_{0},\nu (x_{0})\rangle ^{+}+o(|x-x_{0}|),\quad
\alpha \geq 0,
\end{equation*}
with non-tangential equality in $B,$ and 
\begin{equation*}
u^{-}\leq \beta \langle x-x_{0},\nu (x_{0})\rangle ^{-}+o(|x-x_{0}|),\quad
\alpha \geq 0.
\end{equation*}
with non-tangential equality in $B^{c}.$

We show that $\alpha \leq G(\beta ,\nu (x_{0}),x_{0}).$ Assume by
contradiction that 
\begin{equation*}
\alpha >G(\beta ,\nu (x_{0}),x_{0}).
\end{equation*}%
After a smooth change of variables that flattens the surface ball we may
assume that $B=B_{2}$, $B_{2}^{+}\subset \Omega ,$ with $x_{0}=0\in \partial
\Omega \cap \partial B_{2}$ and $\nu (0)=e_{1}$. We keep the same notation $u
$ and $\mathcal{L}$ for the transformed $u$ and the new operator, which is
uniformly elliptic with ellipticity constant of the same order of $\lambda
,\Lambda $. Let $w\left( x\right) =\varphi \left( x\right) /\varphi
_{x_1}\left( 0\right) $ where $\varphi \left( x\right) $ is the $\mathcal{L}-$%
harmonic measure in $B_{2}^{+}$ of $S_{2}^{+}=\partial B_{2}\cap \left\{
y_{1}>0\right\} $. For $k\geq 1,$ define 
\begin{equation*}
\alpha _{k}=\sup \left\{ \tilde{\alpha}:u\left( x\right) \geq \tilde{\alpha}%
w\left( x\right) \text{ \ \ \ for every }x\in B_{1/k}^{+}\right\} .
\end{equation*}%
Then $\left\{ \alpha _{k}\right\} _{k\in \mathbb{N}}$ is nondecreasing and
(see Lemma 2.3 in \cite{DFS5}), $\alpha _{k}\rightarrow \alpha \geq 0.$
Moreover 
\begin{equation}
u\left( x\right) =\alpha w\left( x\right) +o\left( \left\vert x\right\vert
\right) \text{ \ \ as }x\rightarrow 0\text{, }x\in B_{1}^{+}.  \label{asymm}
\end{equation}%
Thus for each $k$ sufficiently large, there exists a $C^{1,\bar \gamma}\left( 
\overline{B_{1/k}^{+}}\right) $ function $\tilde{\alpha}_{k}w\left( x\right)
,$ that vanishes on $x_{1}=0$ and touches from below $u$ in $B_{1/k}^{+},$
at $0$, with $\tilde{\alpha}_{k}\rightarrow \alpha .$

Let $v\left( x\right) =\psi \left( x\right) /\psi _{x_1}\left( 0\right) ,$
where $\psi $ is the $\mathcal{L-}$ harmonic measure in $B_{2}^{-}$ of $%
\partial B_{2}\cap \left\{ y_{1}<0\right\} $. Define, for $k\geq 1$, 
\begin{equation*}
\beta _{k}=\inf \left\{ \tilde{\beta}:u^{-}\left( x\right) \leq \tilde{\beta}%
v\left( x\right) \text{ \ \ \ for every }x\in B_{1/k}^{-}\right\} .
\end{equation*}%
Then $\left\{ \beta _{k}\right\} _{k\in \mathbb{N}}$ is nonincreasing and
(see Lemma 2.2 in \cite{DFS5}), $\beta _{k}\rightarrow \beta \geq 0.$ Thus,
for each $k$ sufficiently large, there exists a $C^{1,\bar\gamma}\left( 
\overline{B_{1/k}^{-}}\right) $function $\tilde{\beta}_{k}v\left( x\right) ,$
that vanishes on $x=0$ and touches from above $u^{-}$ in $B_{1/k}^{-},$ at $0
$, with $\tilde{\beta}_{k}\rightarrow \beta .$ As a consequence, the
functions 
\begin{equation*}
h_{k}=\left\{ 
\begin{array}{l}
\tilde{\alpha _{k}}w,\quad x\in B_{1/k}^{+} \\ 
-\tilde{\beta}_{k}v,\quad x\in B_{1/k}^{-}%
\end{array}%
\right. 
\end{equation*}%
touch from below $u$ at $0$ in $B_{1/k}$ and therefore we must have $\tilde{%
\alpha _{k}}\leq G(\tilde{\beta}_{k},e_{1},0).$ Since $\tilde{\alpha _{k}}%
\rightarrow \alpha $ and $\tilde{\beta}_{k}\rightarrow \beta ,$ we obtain a
contradiction to $\alpha >G(\beta ,\nu _{0},x_{0}).$
\end{proof}

Finally, to obtain Theorem \ref{precompactenss_th} we need a compactness
result, which is available already for viscosity solutions. We state here a
compactness theorem specific for Perron's solutions, as it is interesting in
its own.

\begin{thm}
\emph{(Compactness).} Let $u_{k}$ be a sequence of Perron solutions to%
\begin{equation*}
\begin{array}{ll}
\mathcal{L}_{k}u_{k}=\textrm{div}\left( A_{k}\left( x\right) \nabla u\right)
=f_{1,k}\text{ } & \text{ \ in }\Omega ^{+}\left( u_{k}\right) \\ 
\mathcal{L}_{k}u_{k}=\text{div}\left( A_{k}\left( x\right) \nabla u\right)
=f_{2,k}\chi _{\left\{ u_{k}<0\right\} }\text{ } & \text{ \ in }\Omega
^{-}\left( u_{k}\right) \\ 
(u_{k}^{+})_{\nu }=G_{k}(\left( u_{k}^{-}\right) _{\nu },\nu_0,x_0) & \text{
\ on }F\left( u_{k}\right)%
\end{array}%
\end{equation*}%
with $\underline{u}_k$ a sequence of minorants s.t. 
\begin{equation*}
\underline{u}_{k}\leq u_{k}.
\end{equation*}
Assume that $A_{k}\rightarrow A$, $G_k\rightarrow G,f_{i,k}\rightarrow f_i,
i=1,2,$ and \underline{$u$}$_{k}\rightarrow \underline{u}$ uniformly and
that the assumptions on $A_k, G_k$,$f_{i,k}, i=1,2,$ and $\underline{u}_k$
are satisfied uniformly.

\noindent Then, if $u_{k}\rightarrow u$ uniformly in $D$, $u$ is a weak
solution of the limiting free boundary problem in $D$.
\end{thm}

\begin{proof}
The proof follows as in Lemma 6.3 and 7.1 in \cite{DFS5}, with the sequence $%
u_{k}$ playing the role of the sequence of supersolutions $w_{k}$ there.
\end{proof}

\begin{proof}[Proof of Theorem $\protect\ref{precompactenss_th}$]
We may suppose that $x_{0}=0$ and that $\nu =e_{n}$ is the
measure-theoretical normal to $F^{\ast }\left( u\right) $ at $0$. Rescale by
setting $u_{r}\left( x\right) =\frac{1}{r}u\left( rx\right) $. Since $\Omega
^{+}\left( u_{r}\right) =\Omega _{r}^{+}\left( u\right) \equiv \left\{
x:rx\in \Omega ^{+}\left( u\right) \right\} ,$ it follows that $\Omega
^{+}\left( u_{r}\right) $ and $\Omega ^{-}\left( u_{r}\right) $ converge
locally in measure to the half spaces $x_{n}>0$ and $x_{n}<0$, respectively.

Moreover, from the uniform positive density of $\Omega ^{+}\left( u\right) $
along $F\left( u\right) ,$ given $\varepsilon >0$ and a ball $B$ centered at
the origin, for $r\leq r_{0}\left( \varepsilon ,B\right) $, small, we infer%
\begin{equation}
\Omega ^{+}\left( u_{r}\right) \cap B\subset \left\{ x_{n}>-\varepsilon
\right\} \cap B\text{ }  \label{ops}
\end{equation}
and therefore also 
\begin{equation}
\Omega ^{-}\left( u_{r}\right) \cap B\supset \left\{ x_{n}<-\varepsilon
\right\} \cap B.  \label{opss}
\end{equation}%
Now, from the equilipschitz continuity of $u_{r}$, we can extract a
subsequence $u_{j}=u_{r_{j}}$ uniformly convergent to $U$ in every
compact subset of $\mathbb{R}^{n}$. From the compactness theorem, $U$ is
a global solution of a two phase problem for the Laplace operator, with free
boundary condition $U_{x_{n}}^{+2}-U_{x_{n}}^{-2}=1$. The above argument and the Lipschitz
continuity of $U$, implies that $U$ must be a two-plane solution:%
\begin{equation*}
U\left( x\right) =\alpha x_{n}^{+}-\beta x_{n}^{-}.
\end{equation*}%
Thus, ($j$ large) 
\begin{equation*}
\left\vert u_{j}-U\right\vert <\varepsilon ,
\end{equation*}%
and using the nondegeneracy of $u$ and (\ref{ops}),(\ref{opss}), it also
follows that $F\left( u_{j}\right) $ is contained in a strip $\left\vert
x_{n}\right\vert \leq c\varepsilon ,$ $c$ universal.

Then, the regularity Theorem \ref{flatmain2} implies that $F\left( u\right) $
is a $C^{1,\gamma}$ surface, near the origin. In particular, $\alpha $ and $%
\beta $ must be independent of the selected subsequence.
% and $\alpha =u_{\nu
%}^{+}=G\left( \beta ,e_{n},0\right) =G\left( u_{\nu }^{-},e_{n},0\right) $.
\end{proof}

\section{Application to the Prandtl-Batchelor model}

In this section we apply our results to the classical Prantl-Batchelor model
in hydrodynamics, proposed by Batchelor back in 1956 \cite{B1}. We restate
it in our notation the following way.

In a bounded 2d domain $\Omega $ a constant vorticity flow is immersed in an
irrotational flow. On the unknown interface between the two flows, the jump
of the squares of the exterior and the interior speeds is a prescribed
constant. Thus given two constants $\mu >0,\omega >0,$ one looks for a
function $u$, with $u=\mu$ on $\partial \Omega,$ satisfying 
\begin{equation}  \label{eqs}
\begin{cases}
\Delta u=0\text{ \ in }\Omega ^{+}\left( u\right), \quad \Delta u=\omega 
\text{ \ in }\Omega ^{-}\left( u\right) \\ 
\  \\ 
\left\vert \nabla u^{+}\right\vert ^{2}-\left\vert \nabla u^{-}\right\vert
^{2}=\sigma \quad \text{on $F\left( u\right) =\Omega \cap \p\Omega ^{+}\left(
u\right) $}%
\end{cases}%
\end{equation}
where $\sigma >0$.

When $\Omega $ is convex, Acker \cite{A}, using a variational method, gives
sufficient conditions for the existence of a solution $U$ with \emph{convex}
free boundary, such as $\Omega$ is large enough or there exists a classical
``inner" solution (i.e., a supersolution in our setting). These solutions
are classical in the sense that $U^{+}$ and $U^{-}$ are $C^{1}$ up to the
free boundary and the condition across is satisfied in the pointwise sense.
Actually, the results of Acker hold in any dimension $n\geq 2,$ with $\sigma
=\sigma \left( x\right) $ continuous and bounded. From Theorem \ref{Lipmainvar}, we deduce that 
$F\left( U\right)$ is locally a $C^{1,\gamma}$ graph.

When $\Omega $ is not convex, the theory is largely incomplete (see \cite{EM}%
). In the context of viscosity solutions it is known that solutions are
Lipschitz, as shown in [CJK], but neither existence nor regularity is known.
What we can prove is the following. Assume that $f$ is a conformal map of $%
\Omega $ onto the unit disk $B_{1}$ centered at the origin, with $0<m\leq
\left\vert f^{\prime }\right\vert \leq M$. Call $\hat{u}=u\circ f$. Problem (%
\ref{eqs}) transforms into%
\begin{equation}  \label{eqs1}
\begin{cases}
\Delta \hat{u}=0\text{ \ in }B_{1}^{+}\left( \hat{u}\right) \text{, }\Delta 
\hat{u}=\left\vert f^{\prime }\right\vert ^{2}\omega \text{ \ in }%
B_{1}^{-}\left( \hat{u}\right) \\ 
\  \\ 
\left\vert \nabla \hat{u}^{+}\right\vert ^{2}-\left\vert \nabla \hat{u}%
^{-}\right\vert ^{2}=\left\vert f^{\prime }\right\vert ^{2}\sigma \quad 
\text{on $F\left( \hat{u}\right) .$}%
\end{cases}%
\end{equation}%
Clearly, if $u$ is a viscosity solution of (\ref{eqs}) then $\hat{u}$ is a
viscosity solution of (\ref{eqs1}).

Let us look for radial solutions $U_{h}$ of problem (\ref{eqs1}), with $%
h^{2}\omega $ and $h^{2}\sigma $ replacing $\left\vert f^{\prime
}\right\vert ^{2}\omega $ and $\left\vert f^{\prime }\right\vert ^{2}\sigma $
respectively, where $h>0$. It turns out that, if $z_{0}=z_{0}\left( \mu
,\omega,h\right) $ is the minimum of the function%
\begin{equation*}
f\left( \rho ;h,\mu \right) =\rho ^{-2}\left( \log \rho \right)
^{-2}\mu ^{2}-\frac{1}{4}h^{2}\omega \rho ^{2}
\end{equation*}%
over the interval $\left( 0,1\right),$ and the condition 
\begin{equation}
h^{2}\sigma >\max \left\{ 0,z_{0}\right\}  \label{cond}
\end{equation}
holds, then it is easy to check that there exist two radial solutions $%
U_{h,1}$ and $U_{h,2}$, whose free boundary is given by circles $C_{\rho
_{1}}$ and $C_{\rho _{2}},$ where $\rho _{1}\left( h,\omega ,\mu ,\sigma
\right) <\rho _{2}\left( h,\omega ,\mu ,\sigma \right) $ are the roots (in
the interval $\left( 0,1\right) $) of the equation%
\begin{equation*}
f\left( \rho ;h,\mu \right) =h^{2}\sigma .
\end{equation*}%
In particular, (\ref{cond}) holds for every $h$ and $\sigma >0$, if $%
h^{2}\omega \geq 4e\mu $.

\emph{\ }Observe now that, under the condition (\ref{cond}), the radial
solutions $U_{m,1}$ and $U_{m,2}$ are classical supersolutions of problem (%
\ref{eqs1}), while $U_{M,1},U_{M,1}$ are classical subsolutions. Since, in
particular 
\begin{equation*}
\rho _{2}\left( M,\omega ,\mu ,\sigma \right) >\rho _{2}\left( m,\omega ,\mu
,\sigma \right) >\rho _{1}\left( m,\omega ,\mu ,\sigma \right)
\end{equation*}%
it follows that $U_{M_{2}}<U_{m,1}<$ $U_{m,2}$ in $B_1$. From Theorem 1.3 in 
\cite{DFS5} and Theorem \ref{precompactenss_th} we deduce the following
result.

\begin{thm}
Assume that 
\begin{equation}
m^{2}\sigma >\max \left\{ 0,z_{0}\left( \mu ,\omega, m\right)
\right\} .  \label{con}
\end{equation}
Then there exists a Perron solution $u$ of the Prandtl-Batchelor problem %
\eqref{eqs}. In particular, the free boundary $F\left( u\right)$ has $H^{1}$
finite measure and in a neighborhood of any point of the reduced boundary, $%
F^*\left( u\right) $ is a $C^{1,\gamma}$ curve.
\end{thm}

\end{document}